\documentclass{amsart}
\pdfoutput=1
\usepackage{setspace}
\usepackage{amssymb,color,enumerate}
\usepackage{latexsym,url}
\usepackage{graphicx}
\usepackage{pdfsync,url}
\usepackage{amsmath}
\usepackage{amsfonts}
\allowdisplaybreaks

\usepackage[square,comma,sort&compress]{natbib}
\setcitestyle{numbers}
\usepackage{mathrsfs} 

\theoremstyle{plain}
 \newtheorem{thm}{Theorem}[section]
 \newtheorem{prop}{Proposition}[section]
 \newtheorem{lem}{Lemma}[section]
 \newtheorem{cor}{Corollary}[section]
\theoremstyle{definition}
 \newtheorem{exm}{Example}[section]
 
 \newtheorem{dfn}{Definition}[section]
\numberwithin{equation}{section}
\renewcommand{\leq}{\leqslant}
\renewcommand{\geq}{\geqslant}
\renewcommand{\setminus}{\smallsetminus}

\newcommand{\bfA}{\mathbb{A}}

\newcommand{\bfS}{\mathbb{S}}
\newcommand{\bfM}{\mathbb{M}}

\newcommand{\bfD}{\mathbb{D}}

\newcommand{\bfC}{\mathbb{C}}

\newcommand{\bfO}{\mathbb{O}}

\newcommand{\R}{\mathbb{R}}

\newcommand{\C}{\mathscr{C}}

\newcommand{\scrS}{\mathscr{S}}

\newcommand{\scrA}{\mathscr{A}}
\newcommand{\scrD}{\mathscr{D}}
\newcommand{\M}{\mathbb{M}}
\newcommand{\bfH}{\mathbb{H}}
\newcommand{\MO}{\mathbb{M}_{\bfO}}

\newcommand{\ud}{\mathrm{d}}
\newcommand{\vep}{\epsilon}
\newcommand{\Prob}{P}
\newcommand{\E}{E}
\newcommand{\es}{\emptyset}
\newcommand{\vague}{\stackrel{\lower0.2ex\hbox{$\scriptscriptstyle
                    \it{v} $}}{\rightarrow}}
\newcommand{\weak}{\stackrel{\lower0.2ex\hbox{$\scriptscriptstyle
                    \it{w} $}}{\rightarrow}}
\newcommand{\what}{\stackrel{\lower0.2ex\hbox{$\scriptscriptstyle
                    \it{\hat{w}} $}}{\rightarrow}}
\newcommand{\eqdis}{\stackrel{\lower0.2ex\hbox{$\scriptscriptstyle
                    \mathrm{d}$}}{=}}
\newcommand{\distr}{\stackrel{\lower0.2ex\hbox{$\scriptscriptstyle
                    \it{d} $}}{\rightarrow}}
\def\dsk{d_{\text{sk}}}
\def\bX{\boldsymbol X}

\def\C{\mathcal{C}}

\def\cumsum{\text{CUMSUM}}
\def\dinfty{d_\infty}
\def\Rplus{\mathbb{R}_+}
\def\polar{\text{POLAR}}
\def\gpolar{\text{GPOLAR}}
\def\projp{\text{PROJ}_p}

\def\cadlag{c\`adl\`ag{ }}
\title[Regularly varying measures]{Regularly Varying Measures on
  Metric Spaces: Hidden Regular Variation and Hidden Jumps}

\subjclass[2010]{28A33,60G17,60G51,60G70}

\keywords{regular variation, multivariate heavy tails, hidden regular
  variation, tail estimation, L\'evy {process}}

\author[Lindskog]{Filip Lindskog}
\address{
Filip Lindskog\\Department of Mathematics \\ 
KTH Royal Institute of Technology,
100 44 Stockholm,
Sweden}
\email{lindskog@kth.se}

\author[Resnick ]{Sidney I. Resnick}
\address{Sidney I. Resnick\\School of ORIE, Cornell University,
Ithaca, NY 14853} \email{sir1@cornell.edu}

\author[Roy ]{Joyjit Roy}
\address{Joyjit Roy\\School of ORIE, Cornell University,
Ithaca, NY 14853} \email{jr653@cornell.edu}

\thanks{S. Resnick and J. Roy were supported by Army MURI grant
  W911NF-12-1-0385 to Cornell University. Resnick acknowledges
  hospitality, space and support from Columbia University during his
  sabbatical year 2012-2013.} 

\begin{document}

\setcounter{page}{1}
\thispagestyle{empty}

\begin{abstract}
  We  develop a framework
for regularly varying measures on complete separable metric
  spaces $\mathbb{S}$ with a closed cone $\mathbb{C}$ removed,
extending  material in \cite{hult:lindskog:2006a,
  das:mitra:resnick:2013}. Our framework provides a flexible
way to consider hidden regular
  variation and allows  simultaneous regular variation properties to
  exist at different scales and provides potential for more accurate
  estimation of probabilities of risk regions. We apply our framework
  to iid random variables in $\mathbb{R}_+^\infty$ with marginal
  distributions having regularly varying tails and to
\cadlag L\'evy processes
whose L\'evy measures have regularly varying
  tails. In both cases, an infinite number of regular variation
  properties coexist distinguished by different scaling functions and
  state spaces.
\end{abstract}
\bibliographystyle{plainnat}
\maketitle

\section{Introduction}
This paper discusses a framework for regular variation and heavy tails
for distributions of metric space valued random elements and applies this framework to
regular variation for measures on $\mathbb{R}^\infty_+$ and $\bfD([0,1],\R)$.

Heavy tails appear in diverse contexts such as risk
management; quantitative finance and economics; complex networks 
of data and telecommunication transmissions; as well as the
rapidly expanding field of social
networks. Heavy tails are also colloquially called power law tails or
Pareto tails, especially in one dimension.  The mathematical formalism
for discussing heavy tails is the theory of regular
variation, originally formulated on $\mathbb{R}_+$ and extended to more
general spaces. See, for instance, \cite{resnickbook:2008,
resnickbook:2007, resnick:1986, seneta:1976,
embrechts:kluppelberg:mikosch:2003, 
bingham:goldie:teugels:1987,
dehaan:1970, dehaan:ferreira:2006,
geluk:dehaan:1987,
hult:lindskog:2006a}.

One approach to estimating the probability of a remote risk region relies
on asymptotic analysis from the theory of extremes or heavy tail
phenomena. Asymptotic methods come with the obligation to choose
an asymptotic regime among potential competing regimes.
This is often
tantamount to choosing a state space for the observed random
elements as well as a scaling. For example, in $\mathbb{R}^2_+$,  for a risk vector
$X=(X_1,X_2)$, 
if we need to estimate $P[ X>x]=P[X_1>x_1,\,X_2>x_2]$ for large
$x$, should the state space for asymptotic analysis be
$[0,\infty]^2\setminus \{(0,0)\}$ or $(0,\infty]^2$?
Ambiguity for the choice of asymptotic regime led to the 
idea of {\it coefficient of tail dependence\/} \citep{peng:1999,
  schlather:2001,
bruun:tawn:1998,
coles:heffernan:tawn:1999,ledford:tawn:1997,ledford:tawn:1996},
{\it hidden regular variation\/} (hrv)
\citep{maulik:resnick:2005,mitra:resnick:2011hrv,
mitra:resnick:2013,
resnick:2002a,
heffernan:resnick:2005, 
  resnick:2008, resnickbook:2007}
and the conditional extreme value (cev) model
\citep{heffernan:tawn:2004,das:resnick:2011b,das:resnick:2011,
  das:mitra:resnick:2013, resnick:zeber:2012kernel}.

Due to the scaling inherent in the definition of regular variation, a
natural domain for regularly varying tails is a 
region closed
under scalar multiplication and usually the domain is a
cone centered at the origin. 
Commonly used cones include
$\mathbb{R}_+$, $\mathbb{R}^d_+$, or the two sided versions allowing
negative values that are natural in finance and economics.
However, as argued in 
\cite{das:mitra:resnick:2013}, there is need for other cones as well,
particularly when asymptotic independence or asymptotic full dependence
(\cite[Chapter 5]{resnickbook:2008}, \cite{sibuya:1960}) is present.
Going beyond finite dimensional spaces, there is a need for a comprehensive theory covering 
spaces such as
$\mathbb{R}^\infty_+$ and function spaces. Fortunately a good
framework for such a theory {of regular variation on metric spaces after removal of a
point} was {created in 
\cite{hult:lindskog:2006a}}. 
  The need to remove more than a point, perhaps a closed set and
certainly a closed cone, was argued in
\cite{das:mitra:resnick:2013}. These ideas build on
 $w^\#$-convergence in \citep[Section A2.6]{daley:vere-jones:1988}.

This paper has a number of goals:
\begin{enumerate}
\item We follow the lead of {
\cite{hult:lindskog:2006a} and}
develop a theory of regularly varying measures on complete separable
metric spaces $\mathbb{S}$ with a closed cone $\mathbb{C}$ 
removed.   Section \ref{secconv} develops
a topology on the space of measures on $\mathbb{S}\setminus
\mathbb{C}$ which are finite on regions at positive distance from
$\mathbb{C}$. This topology allows creation of  mapping
theorems (Section \ref{subsec:mappingTheorems})
that encourage continuity arguments and is designed to allow simultaneous 
regular variation properties to exist at different scales as is
considered in hidden regular variation.

\item We apply the {general material of
Section \ref{secconv}  to two significant applications.}
\begin{enumerate}
\item In Section \ref{seqSpace} we focus on $\R_+^p$ and 
  $\mathbb{R}^\infty_+$, the space of sequences with non-negative
  components.  An iid sequence $X=(X_1,X_2,\dots ) \in
  \mathbb{R}^\infty_+$ such that $P[X_1>x]$ is regularly varying has a
  distribution which is regularly varying on $  \mathbb{R}^\infty_+
  \setminus \bfC_{\leq j}$ for any $j\geq 1$, where
$\bfC_{\leq j}$ are sequences with at most $j$ positive
components. Mapping theorems (Section \ref{subsec:mappingTheorems}) allow extension to
the regular variation properties of $S=(X_1,X_1+X_2, X_1+X_2 + X_3, \dots)$ in $
\mathbb{R}^\infty_+$ minus the set of non-decreasing sequences which
are constant after the $j$th component. See Section \ref{subsubsec:HRV}.
For reasons of simplicity and taste, we
restrict discussion to $\mathbb{R}^\infty_+$ but with {modest}
 effort, results could be extended to $\mathbb{R}^\infty$.
We also discuss regular variation of the distribution of a sequence of
Poisson points in $\R_+^\infty$ (Section \ref{subsubsec:PoissonPts}).
\item The $\R_+^\infty$ discussion of Poisson points in Section \ref{subsubsec:PoissonPts}
can be leveraged in a natural way to consider (Section \ref{sec:hrvLevy}) 
 regular variation of the distribution of a L\'evy
  process whose L\'evy measure $\nu$
is regularly varying: $\lim_{t\to\infty} t\nu\bigl(b(t)x,\infty \bigr)
=x^{-\alpha},\,x>0$, for some scaling function $b(t) \to\infty$.
We reproduce the result
{\cite{hult:lindskog:2005SPA,hult:lindskog:2007}} 
that
the limit measure of regular variation with scaling $b(t)$ on $\bfD([0,1],\R)\setminus \{0\}$ 
concentrates on \cadlag functions with {\it one\/} positive jump. This raises the natural question of what happened to the rest of the jumps of the L\'evy process that seem to be 
{\it hidden\/} by the scaling $b(t)$.  
We are able to  generalize for any $j\geq 1$ to convergence under
the weaker normalization $b(t^{1/j})$ on a smaller space in which the limit measure
concentrates on non-decreasing functions with $j$ positive jumps.
Again, as in the study of $\mathbb{R}_+^\infty$, we focus for simplicity only on large
positive jumps of the L\'evy process.
\end{enumerate}
\item A final goal is to clarify the proper definition of regular
  variation in metric spaces. For historical reasons, regular variation is usually associated
  with scalar multiplication but what does this mean in a general
  metric space? Traditional
 definitions are in  Cartesian coordinates in finite dimensional spaces and the form
  of the definition may not survive change of coordinates. For
  example, in $\mathbb{R}^p_+$, a random vector $X$ (in Cartesian
  coordinates) has a regularly varying distribution if for some
  scaling function $b(t)\to\infty$ we have $tP[X /b(t) \in \cdot \,] $
  converging to a limit. If we transform to polar coordinates
  $X\mapsto (R,\Theta):=(\|X\|, X/\|X\|)$, the
limit is taken on 
  $tP[(R/b(t), \Theta) \in \cdot \,]$ which appears to be subject to
  a different notion of scaling. 
The two
  convergences are equivalent but look different unless one allows for
  a more flexible definition of scalar multiplication. We discuss
  requirements for scalar multiplication in Section \ref{secconv}
  along with some examples; related
  discussion is in \cite{
balkema:embrechts:2007,
balkema:1973,
meerschaert:scheffler:2001}.
\end{enumerate}

The existing theory for regular variation 
on, say, $\mathbb{R}^d_+$, uses the set-up of vague convergence.
A troubling consequence is the need to use the one point
uncompactification \cite[page 170ff]{resnickbook:2007}  which adds
lines through infinity to the state space. When regular
variation is defined on the cone $[0,\infty]^d \setminus \{0\}$, 
limit measures cannot charge 
lines through infinity. However, on proper subcones of $[0,\infty]^d
\setminus \{0\}$ this
is no longer true and this creates some mathematical havoc:
Convergence to types arguments can fail and limit measures may not be
unique: In given examples \citep[Example
5.4]{das:mitra:resnick:2013}, under one normalization  the limit measure
concentrates on lines through infinity and under another it
concentrates on  finite points.
Another difficulty is that the polar coordinate transform 
$x \mapsto (\|x\|, x / \|x\|)$ 
cannot be defined on lines through infinity. One way to patch 
things up is to retain the one point un-compactification but 
demand all limit measures have no mass on
lines through infinity but this 
does not resolve all difficulties since
the unit sphere $\{x: \|x\|=1\}$ 
defined by the norm $x\mapsto \|x\|$ may not be
compact on a subcone such as $(0,\infty]^d$. Another way forward which we deem cleaner and
more suitable to general spaces where
compactification is more involved, is not to compactify and just to
define tail regions as subsets of the metric space at positive distance from 
the deleted closed set. This is the approach given in Section \ref{secconv}.

\section{Convergence of measures in the space $\MO$}\label{secconv}

Let $(\mathbb{S}, d)$ be a complete separable metric space. 
The open ball centered at $x \in \mathbb{S}$ with radius
$r$ is written
$B_{x,r}=\{y \in \bfS : d(x,y) < r\}$ and these open sets generate 
$\scrS$, the Borel $\sigma$-algebra on $\bfS$.
 For $A \subset \bfS$, let $A^\circ$ and $A^-$ denote the interior
and closure of $A$, respectively, and let $\partial A = A^-\setminus
A^\circ$ be the boundary of $A$. 
Let $\C_b$ denote the class of real-valued, non-negative, bounded and continuous
functions on $\bfS$, and let $\M_b$ denote the class of finite Borel
measures on $\scrS$. A basic neighborhood of $\mu \in \M_b$ is a set
of the form $\{\nu \in \M_b :|\int f_i \ud\nu -\int f_i \ud\mu| <
\vep, i = 1,\dots,k\}$, where $\vep > 0$ and $f_i \in \C_b$ for
$i=1,\dots,k$. Thus a sub-basis for  $\M_b$ are sets of the form $\{\nu
\in \M_b: \nu(f):=\int f d\nu \in G\}$ for $f\in \C_b$ and $G$ open in $\mathbb{R}_+$.
This equips $\M_b$ with the weak topology and
convergence $\mu_n \to \mu$ in $\M_b$ means
 $\int f \ud \mu_n \to \int f \ud \mu$ for all $f \in \C_b$. See
e.g.~Sections 2 and 6 in \cite{billingsley:1999} for details. 

Fix a closed set $\bfC \subset \bfS$ and set $\bfO=\bfS\setminus
\bfC$, e.g. one possible choice is $\bfO=\bfS\setminus \{s_0\}$ for
$\bfC = \{s_0\}$ for some $s_0\in\bfS$. The subspace $\bfO$ is a
metric subspace of $\bfS$ in the relative topology with
$\sigma$-algebra {$\scrS_{\bfO}=\scrS(\bfO)=\{A : A \subset \bfO, A \in \scrS\}$}.  

Let $\C_{\bfO} =\C(\bfO)$ denote the real-valued, non-negative, bounded and continuous functions $f$ on $\bfO$ such that for each $f$ there exists $r>0$ 
such that $f$ vanishes on $\bfC^r$; we use the notation $\bfC^r=\{x\in
\bfS:d(x,\bfC)<r\}$, where $d(x,\bfC)=\inf_{y\in \bfC}d(x,y)$.
Similarly, we will write
$d(A,\bfC)=\inf_{x \in A,\,y\in \bfC}d(x,y)$ for $A\subset\mathbb{S}$.
We say that a set
$A\in\scrS_{\bfO}$ is bounded away from $\bfC$ if $A \subset \bfS
\setminus\bfC^r$ for some $r > 0$ or equivalently $d(A,\bfC)>0$.
So $\C_{\bfO}$ consists of non-negative continuous functions
whose supports are bounded away from $\bfC$.
Let $\MO$ be the class of Borel measures on $\bfO$ whose restriction
to $\bfS \setminus\bfC^r$ is finite for each $r > 0$. When convenient,
we also write $\mathbb{M}(\mathbb{O})$ or
$\mathbb{M}(\mathbb{S}\setminus \bfC)$.  A basic
neighborhood of $\mu \in \MO$ is a set of the form $\{\nu \in \MO
:|\int f_i \ud \nu - \int f_i \ud \mu| < \vep, \,i = 1,\dots,k\}$, where
$\vep > 0$ and $f_i \in \C_{\bfO}$ for $i=1,\dots,k$. 
A sub-basis is formed by sets of the form 
\begin{align}\label{eq:sub-bas}
\{\nu \in \MO: \nu(f) \in G\}, \quad f \in
\C_{\bfO}, \quad G \text{ open in }\mathbb{R}_+.
\end{align}
Convergence $\mu_n\to\mu$ in $\MO$ is convergence in the topology defined by this base or sub-base. 

For $\mu \in \MO$ and $r>0$, let $\mu^{(r)}$ denote the restriction of $\mu$ to $\bfS \setminus \bfC^r$. Then $\mu^{(r)}$ is finite  and $\mu$ is uniquely determined by its restrictions $\mu^{(r)}$, $r > 0$. Moreover, convergence in $\MO$ has a natural characterization in terms of weak convergence of the restrictions to $\bfS \setminus \bfC^r$.

\begin{thm}[{\bf Portmanteau theorem}]\label{portthm}
Let $\mu,\mu_n \in \MO$.
The following statements are equivalent.
\\
(i) $\mu_n \to \mu$ in $\MO$ as $n\to\infty$.
\\
(ii)  $\int f \ud\mu_n \to \int f \ud \mu$ for each $f \in \C_{\bfO}$
which is also uniformly continuous on $\bfS$. 
\\
(iii) $\limsup_{n\to\infty}\mu_n(F)\leq\mu(F)$ and $\liminf_{n\to\infty}\mu_n(G)\geq\mu(G)$
for all closed $F\in \scrS_{\bfO}$ and open $G\in \scrS_{\bfO}$ and
$F$ and $G$ are bounded away from $\bfC$.
\\
(iv) 
$\lim_{n\to\infty}\mu_n(A)=\mu(A)$ for all $A\in\scrS_{\bfO}$ bounded away from $\bfC$
with $\mu(\partial A)=0$.
\\
{(v)} $\mu_n^{(r)}\to\mu^{(r)}$ in $\M_b(\bfS \setminus \bfC^r)$ for all but at most countably many $r>0$.
\\
(vi) There exists a sequence $\{r_i\}$ with $r_i \downarrow 0$ such that $\mu_n^{(r_i)}\to\mu^{(r_i)}$ in $\M_b(\bfS \setminus \bfC^{r_i})$ for each $i$.
\end{thm}

For proofs, see Section \ref{secproofs}. Note, the result is
true for any general metric space. 

Weak convergence is metrizable (for instance by the Prohorov metric;
see e.g.~p.~72 in \cite{billingsley:1999}) 
and the close relation between weak convergence and convergence in $\MO$ in Theorem
\ref{portthm}(v)-(vi) indicates that the topology in $\MO$ is
metrizable too. 
{With minor modifications of the}
arguments in \cite{daley:vere-jones:1988}, pp.~627-628, we may choose
the metric 
\begin{align}\label{M0metric}
  d_{\MO}(\mu,\nu)=\int_0^{\infty}e^{-r}p_r(\mu^{(r)},\nu^{(r)})[1+p_r(\mu^{(r)},\nu^{(r)})]^{-1}\ud r,
\end{align}
where $\mu^{(r)},\nu^{(r)}$ are the finite restriction of $\mu,\nu$ to $\bfS \setminus \bfC^r$, and $p_r$ is the Prohorov metric on $\M_b(\bfS \setminus \bfC^r)$.
\begin{thm}\label{M0metrizable}
  $(\MO,d_{\MO})$ is a separable and complete metric space.
\end{thm}

\subsection{Mapping theorems.}\label{subsec:mappingTheorems}
Applications of weak convergence often rely on continuous mapping
theorems and we present versions for convergence in
$\MO$.  
Consider another separable and complete metric space $\bfS'$ and let $\bfO',\scrS_{\bfO'},\bfC',\M_{\bfO'}$ have the same meaning relative to the space $\bfS'$ as do $\bfO,\scrS_{\bfO},\bfC,\M_{\bfO}$ relative to $\bfS$.

\begin{thm}[Mapping theorem]\label{mapthm}
Let $h: (\bfO,\scrS_{\bfO}) \mapsto (\bfO',\scrS_{\bfO'})$ be a
measurable mapping such that $h^{-1}(A')$ is bounded away from $\bfC$
for any $A'\in\scrS_{\bfO'} \cap h(\bfO)$ bounded away from $\bfC'$, and 
$\mu(D_h)=0$, where $D_h$ is the set of discontinuity points of $h$.
Then $\hat h: \bfM_{\bfO} \mapsto \bfM_{\bfO'} $ defined by $\hat
h(\mu)=\mu {\circ} h^{-1}$ is continuous.
\end{thm}


This result is illustrated in Examples \ref{ex:mappingthm}, \ref{ex:mappingthm} and
\ref{eg:polar} and is also needed for considering the generalized polar
coordinate transformation in Section \ref{subsub:polarCont}.  It is
the basis for the approach to regular variation of L\'evy processes in
Section \ref{sec:hrvLevy}.
Theorem \ref{mapthm} is formulated so that $h$ is defined on
$\bfO=\mathbb{S}\setminus \bfC$, rather than on all of $\mathbb{S}$. If
$\bfS=\mathbb{R}_+^p$ and $h(x)=(\|x\|, x/\|x\|)$ is the polar
coordinate transform, then $h$ is not defined at $0$. This lack
of definition is not a problem since
\begin{align*}
h:\bfO:=&\mathbb{R}_+^p
\setminus \{0\} \mapsto \bfO':=(0,\infty)\times \{x \in \mathbb{R}_+^p: \|x\|=1\} \\
=&[0,\infty)\times  \{x \in \mathbb{R}_+^p: \|x\|=1\} \setminus
\Bigl(\{0\}\times  \{x \in \mathbb{R}_+^p: \|x\|=1\}\Bigr).
\end{align*}

The proof is in Section \ref{subsubsec:mapthm} but it is instructive
to quickly consider the special case
where $D_h=\emptyset$ so that $h$ is continuous. In this case $h$
induces a continuous mapping $\hat h:\mathbb{M}_\bfO \mapsto \bfM'_{\bfO'}$
defined by $\hat h(\mu)=\mu \circ h^{-1}$. To see this, look at the
inverse image of a sub-basis set \eqref{eq:sub-bas}: For $G$ open in $\mathbb{R}_+$, and
$f' \in \C_{\bfO'}$,
\begin{align*}
\hat h^{-1}\{\mu' \in \mathbb{M}_{\bfO'} : \mu'(f')\in G\}=&
\{\mu  \in \mathbb{M}_{\bfO} : \mu \circ h^{-1}(f')\in G\}\\=&
\{\mu  \in \mathbb{M}_{\bfO} : \mu (f' \circ h)\in G\}.
\end{align*}
Since $h$ is continuous and $f' \circ h \in \C_{\bfO}$,
$\{\mu  \in \mathbb{M}_{\bfO} : \mu (f' \circ h)\in G\}$ is open in $\mathbb{M}_\bfO$.

Here are two  variants of the mapping theorem. The first allows
application of the operator taking successive partial sums from
$\R^\infty_+ \mapsto \R^\infty_+$ in Proposition \ref{prop:cumsumCont}
and also allows application of the projection map $(x_1,x_2,\dots
)\mapsto (x_1,\dots,x_p)$ from $\R^\infty_+ \mapsto \R^p_+$ in
Proposition \ref{prop:projCont}.
The second variant allows a quick proof that the polar coordinate
transform is continuous on $\R^p_+ \setminus \{0\}$ in Corollary \ref{cor:polarConv}.

\begin{cor}\label{cor:variant1}
Suppose $h:\bfS \mapsto \bfS'$ is uniformly continuous and $\bfC':=h(\bfC)$ is closed in $\bfS'$.
Then $\hat h: \bfM_{\bfO} \mapsto \bfM_{\bfO'} $ defined by $\hat
h(\mu)=\mu {\circ } h^{-1}$ is continuous. 
\end{cor}


\begin{cor}\label{cor:2ndVariant}
Suppose $h:\bfS \mapsto \bfS'$ is continuous and either $\bfS$ or $\bfC$ is compact. Then 
 $\hat h: \bfM_{\bfO} \mapsto \bfM_{\bfO'}$ defined by $\hat
 h(\mu)=\mu { \circ} h^{-1}$ is continuous.
\end{cor}

\subsection{Relative compactness in $\M_{\bfO}$}\label{subsec:rc}
Proving convergence sometimes requires a characterization of relative
compactness.
A subset of a
topological space is relatively compact if its closure is
compact. A subset of a metric space is compact if and only if it is
sequentially compact. Hence, $M \subset \MO$ is relatively compact if
and only if  every sequence $\{\mu_n\}$ in $M$ contains a convergent
subsequence.   
For $\mu \in M \subset \MO$ and $r > 0$, let $\mu^{(r)}$ be
the restriction of $\mu$ to $\bfS \setminus \bfC^r$ and
$M^{(r)}=\{\mu^{(r)} : \mu \in M\}$. By Theorem \ref{portthm} (vi) we
have the following characterization of relative compactness.

\begin{thm}\label{rc0}
A subset $M \subset \MO$ is relatively compact if and only if there exists a sequence $\{r_i\}$ with $r_i \downarrow 0$ such that $M^{(r_i)}$ is relatively compact in $\M_b(\bfS \setminus \bfC^{r_i})$ for each $i$. 
\end{thm}

Prohorov's theorem characterizes relative compactness in the weak
topology.  This translates to a characterization
of relative compactness in $\MO$.  

\begin{thm} \label{rc} $M \subset \MO$ is relatively compact if and only if there exists a sequence $\{r_i\}$ with $r_i \downarrow 0$ such that for each $i$ 
  \begin{align}\label{rc1}
    \sup_{\mu \in M} \mu(\bfS \setminus \bfC^{r_i}) < \infty, 
  \end{align}
  and for each $\eta>0$ there exists a compact set $K_i\subset \bfS \setminus \bfC^{r_i}$ such that
  \begin{align}\label{rc2}
    \sup_{\mu\in M}\mu(\bfS \setminus (\bfC^{r_i}\cup K_i)) \leq \eta.
  \end{align}
\end{thm}

\subsection{$\bfM$-convergence vs vague convergence.}\label{subsec:MvsVague}
Vague convergence complies with the topology on the space of measures
which are finite on compacta. Regular variation for measures on a space
such as $\R^p_+$ has traditionally been formulated using vague
convergence after compactification of the space. In order to make use
of existing regular variation theory on $\R_+^p$, it is useful to
understand how $\bfM$-convergence is related to vague convergence.

Let $\bfS$ be a complete separable metric space and suppose
$\bfC $ is closed in $\bfS$. Then $\bfM_+ (\bfS\setminus \bfC) $ is
the collection of measures finite on $\mathcal{K}(\bfS\setminus
\bfC)$, the  compacta of $\bfS\setminus \bfC$:
$$\bfM_+ (\bfS\setminus \bfC) =\{\mu: \mu(K )<\infty,\,\forall K
\in \mathcal{K}(\bfS\setminus \bfC)\}.$$
Vague convergence on $\bfM_+ (\bfS\setminus \bfC) $
 means $\mu \mapsto \mu(f)$ is continuous for all
$f\in \C_K^+ (\bfS\setminus \bfC)$, the continuous functions with
{\it compact\/} support. The spaces $\bfM_+ (\bfS\setminus \bfC) $
and $\bfM (\bfS\setminus \bfC) $ are not the same. For example if
$\bfS=[0,\infty)$ and $\bfC=\{0\}$, $\mu \in \bfM (\bfS\setminus \bfC)
$ means $\mu(x,\infty)<\infty$ for $x>0$ but $\mu \in \bfM_+
(\bfS\setminus \bfC) $ means $\mu ([a,b])<\infty$ for
$0<a<b<\infty$. For instance Lebesgue measure is in
$\bfM_+ (\bfS\setminus \bfC) $
but not in $\bfM (\bfS\setminus \bfC) $. 

\subsubsection{Comparing $\bfM$ vs $\bfM_+$.}\label{subsub:compareSame} 
We have the following comparison.

\begin{lem} $\bfM$-convergence implies vague convergence and 
\begin{equation}\label{eq:compare}
\bfM (\bfS\setminus \bfC) \subset \bfM_+ (\bfS\setminus \bfC) ,\quad
\C_K^+(\bfS\setminus \bfC) \subset \C(\bfS\setminus \bfC).
\end{equation}
\end{lem}

\begin{proof}
If $f\in \C_K^+(\bfS\setminus \bfC) $, its compact support $K \subset
\bfS\setminus \bfC$ must be bounded away from $\bfC$ and hence
$d(K,\bfC)>0$ and $f\in \C(\bfS\setminus \bfC).$ If $\mu \in \bfM
(\bfS\setminus \bfC),$ and $\bfD $ satisfies $d(\bfD,\bfC)>0$, then
$\mu(\bfD)<\infty$. If $K \in \mathcal{K}(\bfS\setminus \bfC)$ then
$d(K,\bfC)>0$ and so $\mu(K) <\infty$, showing any $\mu \in \bfM
(\bfS\setminus \bfC)$ is also in $\bfM_+ (\bfS\setminus \bfC) .$
\end{proof}

{\bf Remark}: Let  
$\bfS=[0,\infty)$ and $\bfC=\{0\}$, and 
$$\mu_n =\sum_{i=1}^{n^2} {\epsilon_{i/n}} \in 
\bfM (\bfS\setminus \bfC) \subset \bfM_+ (\bfS\setminus \bfC).$$
Here and elsewhere, we use the notation ${\epsilon_{x}}$ for the Dirac measure concentrating mass 1 on the point $x$ so that
$\epsilon_x(A) =1$,  if $x\in A$, and  $\epsilon_x(A) =0$, if $x\in A^c$.
We have $\mu_n$ converging to Lebesgue measure in  $\bfM_+ (\bfS\setminus \bfC)$ but $\{\mu_n\}$ does
not converge in  $ \bfM (\bfS\setminus \bfC)$. If $f$ is $0$ on
$(0,1),$ linear on $(1,2)$, $f(2)=1$ and $f$ is constant on
$(2,\infty)$, then $f \in \C (\bfS\setminus \bfC) $ but
$\mu_n (f) \geq \sum_{i={2n+1}}^{n^2} 1 =n^2-2n \to\infty.$

\subsection{Proofs}\label{secproofs}
\subsubsection{Preliminaries.}\label{subsubsec:pre}
We begin with two well known
preliminary lemmas  in topology. The second one is just a version of
Urysohn's lemma 
\cite{
dudley:1989, royden:1988}
for metric spaces.

\begin{lem}\label{lem:distance}
Fix a set $B \subset \bfS$. Then
\\
(i) $d(x, B)$ is an uniformly continuous function in $x$.
\\
(ii) $d(x, B) = 0$ if and only if $x \in B^{-}$.
\end{lem}
\begin{proof}
(i) follows from the following generalization of the triangle inequality. For $x,y \in \bfS$,
\begin{align*}
d(x, B) \leq d(x,y) + d(y, B).
\end{align*}
(ii) is an easy deduction from the definiton of $d(x,B)=\inf_{z\in B}d(x,z)$.
\end{proof}

\begin{lem}\label{lem:urysohn}
For any two closed sets $A,B \subset \bfS$ such that $A \cap B = \es $, there exists a uniformly continuous function $f$ from $\bfS$ to $[0,1]$ such that $f \equiv 0$ on $A$ and $f \equiv 1$ on $B$ 
\end{lem}
\begin{proof}
Define the function $f$ as
\begin{align*}
f(x)= \frac{ d(x,A) } { d(x,A) + d(x,B) }.
\end{align*}
The desired properties of $f$ are easily checked from Lemma \ref{lem:distance}.
\end{proof}

\begin{lem}\label{lem:uncountableunion}
If $A\in\scrS_{\bfO}$ is bounded away from $\bfC$, $A=\cup_{i\in I} A_i$ for an uncountable index set $I$, disjoint sets $A_i\in\scrS_{\bfO}$, and $\mu(A)<\infty$, then $\mu(A_i)>0$ for at most countably many $i$.
\end{lem}
\begin{proof}
Suppose there exists a countably infinite set $I_n$ such that $\mu(A_i)>1/n$ for $i\in I_n$. Then
\begin{align*}
\infty=\sum_{i\in I_n}\mu(A_i)=\mu\Big(\cup_{i\in I_n}A_i\Big)\leq \mu(A)
\end{align*}
which is a contradiction to the assumption that $\mu(A)<\infty$. The conclusion follows from letting  $n\to\infty$.
\end{proof}

\begin{lem}\label{newremii}
For any $\mu \in \MO$,
$\mu(\partial (\bfS\setminus\bfC^{\delta}))>0$ for at most countably many $\delta>0$.
\end{lem}
\begin{proof}
Notice first that { $\partial (\bfS\setminus\bfC^{\delta})=\{x\in\bfS:d(x,\bfC)=\delta\}$ so $\partial (\bfS\setminus\bfC^{\delta_1})\cap \partial (\bfS\setminus\bfC^{\delta_2})=\es$ for $\delta_1\neq\delta_2$ }. The conclusion follows from Lemma \ref{lem:uncountableunion}.
\end{proof}

\subsubsection{Proof of Theorem \ref{portthm}}
We show that (i) $\Rightarrow$ (ii), (ii) $\Rightarrow$ (iii), (iii) $\Rightarrow$ (iv), (iv) $\Rightarrow$ (v),(v) $\Rightarrow$ (vi) and (vi) $\Rightarrow$ (i).

Suppose that (i) holds. Suppose $\mu_n \to \mu$ in $\MO$ and take $f \in\C_{\bfO}$. Given $\vep > 0$ consider the neighborhood $N_{\vep,f}(\mu)=\{\nu: |\int f \ud \nu - \int f \ud \mu| < \vep\}$. By assumption there exists $n_0$ such that $n \geq n_0$ implies $\mu_n \in N_{\vep,f}(\mu)$, i.e.~$|\int f \ud\mu_n - \int f \ud \mu| < \vep$. Hence $\int f \ud \mu_n \to \int f \ud\mu$. 

Suppose that (ii) holds. Take any closed $F$ that is bounded away from $\bfC$. Then there exists $r>0$ such that $F \subset \bfS \setminus \bfC^{r}$. So for all $x \in F$, $d(x,\bfC) \geq r$. So if we define $F^{\vep} = \{x \in \bfS : d(x,F) < \vep \}$, then each $F^{\vep}$ is open, $ F \subset F^{\vep}$ and $F^{\vep} \downarrow F$ as $\vep \downarrow 0$. Also for $\vep < r/2$, we have that for all $x \in F^{\vep}$ $d(x,\bfC) \geq r- r/2 = r/2$, meaning that $F^{\vep} \subset \bfS \setminus \bfC^{r/2}$. For $\vep > 0$,  $\bfS \setminus F^{\vep}$ is closed and {$F \cap (\bfS \setminus F^{\vep}) = \es$}. So for $ 0 < \vep < r/2$, by Lemma \ref{lem:urysohn}, there exists a uniformly continuous function $f$ from $\bfS$ to $[0,1]$ such that $f \equiv 0$ on $\bfS \setminus F^{\vep}$ and $f \equiv 1$ on $F$. Observe that $f \in \C_{\bfO}$ as $F^{\vep} \subset \bfS \setminus \bfC^{r/2}$. So we have
\begin{align*}
\limsup_{n \to \infty} \mu_n(F) \leq \lim_{n \to \infty} \int f \ud \mu_n =  \int f \ud \mu \leq \mu(F^{\vep}).
\end{align*}
As $\vep \downarrow 0$,  $F^{\vep} \downarrow F$ and as $F$ is closed, we have  $\mu(F^{\vep}) \downarrow \mu(F)$. This leads to 
\begin{align*}
\limsup_{n \to \infty} \mu_n(F) \leq  \mu(F).
\end{align*}
Now take any open $G$ bounded away from $\bfC$. Then there exists $r>0$ such that $G \subset \bfS \setminus \bfC^{r}$. So if we define $G_{\vep} = \bfS \setminus \{x \in \bfS \setminus G : d(x,\bfS \setminus G) < \vep \}$, then each $G_{\vep}$ is closed, $ G_{\vep} \subset G$ and {$G_{\vep} \uparrow G$} as $\vep \downarrow 0$. So by Lemma \ref{lem:urysohn}, there exists a uniformly continuous function $f$ from $\bfS$ to $[0,1]$ such that $f \equiv 0$ on $\bfS \setminus G$ and $f \equiv 1$ on $G_{\vep}$. Observe that $f \in \C_{\bfO}$ as $G \subset \bfS \setminus \bfC^{r}$. So we have
\begin{align*}
\liminf_{n \to \infty} \mu_n(G) \geq \lim_{n \to \infty} \int f \ud \mu_n =  \int f \ud \mu \geq \mu(G_{\vep}).
\end{align*}
As $\vep \downarrow 0$,  {$G_{\vep} \uparrow G$} and as $G$ is open, we have  {$\mu(G_{\vep}) \uparrow \mu(G)$}. This leads to 
{
\begin{align*}
\liminf_{n \to \infty} \mu_n(G) \geq  \mu(G).
\end{align*}
}
This completes the proof of (iii).

Suppose that (iii) holds and take $A \in \scrS_{\bfO}$ bounded away from $\bfC$ with $\mu(\partial A)=0$. 
\begin{align*}
  \limsup_{n\to\infty}\mu_n(A) 
  &\leq\limsup_{n\to\infty}\mu_n(A^-)\leq
  \mu(A^-)\\
  &=\mu(A^\circ)
  \leq \liminf_{n\to\infty}\mu_n(A^\circ)
  \leq \liminf_{n\to\infty}\mu_n(A).
\end{align*}
Hence, $\lim_{n\to\infty}\mu_n(A)=\mu(A)$, so that (iv) holds.

Suppose that (iv) holds and take $r>0$ such that $\mu(\partial (\bfS\setminus\bfC^{r}))=0$. By Lemma \ref{newremii}, all but at most countably many $r>0$ satisfy this property. As $\bfS\setminus\bfC^{r}$ is trivially bounded away from $\bfC$, we have that $\mu_n(\bfS\setminus\bfC^{r}) \to \mu(\bfS\setminus\bfC^{r})$. Now any $A \subset \bfS\setminus\bfC^{r}$ is also bounded away from $\bfC$ and as $\bfS\setminus\bfC^{r}$ is closed, $ \partial_{\bfS\setminus\bfC^{r}} A = \partial A$, where the first expression denotes the boundary of $A$ when considered as a subset of $\bfS\setminus\bfC^{r}$. So for any subset $A \subset \bfS\setminus\bfC^{r}$ with $ \mu( \partial_{\bfS\setminus\bfC^{r}} A) = 0$, we have by (iv) that $\mu_n(A) \to \mu(A)$ and hence $\mu^{(r)}_n(A) \to \mu^{(r)}(A)$. The Portmanteau theorem for weak convergence implies $\mu^{(r)}_n \to \mu^{(r)}$ in $\M_b(\bfS\setminus\bfC^r)$. This completes the proof of (v).

Suppose that (v) holds. Since, $\mu_n^{(r)}\to\mu^{(r)}$ in $\M_b(\bfS \setminus \bfC^r)$ for all but at most countably many $r>0$ we can always choose a sequence $\{r_i\}$ with $r_i \downarrow 0$ such that $\mu_n^{(r_i)}\to\mu^{(r_i)}$ in $\M_b(\bfS \setminus \bfC^{r_i})$ for each $i$.

Suppose that (vi) holds. Take $\vep > 0$ and a neighborhood $N_{\vep,f_1,\dots,f_k}(\mu)=\{\nu: |\int f_j\ud\nu-\int f_j\ud\mu|<\vep,j= 1,\dots,k\}$ where each $f_j \in \C_{\bfO}$ for $ j=1,2,\ldots,k$. Let $r > 0$ be  such that $\mu_n^{(r)}\to\mu^{(r)}$ in $\M_b(\bfS \setminus \bfC^{r})$ and each $f_j$ vanishes on $\bfC^{r}$.  Let $n_j$ be an integer such that $n \geq n_i$ implies $|\int f_i \ud \mu_n^{(r)} - \int f_i \ud
\mu^{(r)}| < \vep$. Hence, $n \geq \max(n_1, \dots, n_k)$ implies that  $|\int f_i \ud \mu_n^{(r)} - \int f_i \ud
\mu^{(r)}| < \vep$ for all $j=1,2,\ldots,k$. As each $f_j$ vanishes outside $\bfC^{r}$, we also have that  $|\int f_i \ud \mu_n - \int f_i \ud
\mu| < \vep$ for all $j=1,2,\ldots,k$. So $\mu_n \in N_{\vep, f_1, \dots, f_k}(\mu)$. Hence $\mu_n \to \mu$ in $\MO$. 
\qed

\subsubsection{Proof of Theorem \ref{M0metrizable}.}\label{subsubsec:M0metrizable}
The proof consists of minor modifications of arguments that can be found 
in \cite{daley:vere-jones:1988}, pp.~628-630. Here we change from $r$ to $1/r$. For the
sake of completeness we have included a full proof. 

We show that (i) $\mu_n\to\mu$ in $\MO$ if and only if $d_{\MO}(\mu_n,\mu)\to 0$, and (ii) $(\MO,d_{\MO})$ is separable and complete.

(i) Suppose that $d_{\MO}(\mu_n,\mu)\to 0$. The integral expression in \eqref{M0metric} can be written $d_{\MO}(\mu_n,\mu)=\int_0^\infty e^{-r}g_n(r)\ud r$, so that for each $n$, $g_n(r)$ decreases with $r$ and is bounded by $1$. Helly's selection theorem (p.~336 in \cite{billingsley:1995}), applied to $1-g_n$, implies that there exists a subsequence $\{n'\}$ and a nonincreasing function $g$ such that $g_{n'}(r) \to g(r)$ for all continuity points of $g$. By dominated convergence, $\int_0^{\infty}e^{-r}g(r)\ud r=0$ and since $g$ is monotone this implies that $g(r)=0$ for all finite $r>0$. Since this holds for all convergent subsequences $\{g_{n'}(r)\}$, it follows that $g_n(r)\to 0$ for all continuity points $r$ of $g$, and hence, for such $r$, $p_r(\mu_n^{(r)},\mu^{(r)})\to 0$ as $n\to\infty$. By Theorem \ref{portthm} (vi), $\mu_n\to\mu$ in $\MO$.

Suppose that $\mu_n\to\mu$ in $\MO$. Theorem \ref{portthm} (v) implies that $\mu_n^{(r)}\to\mu^{(r)}$ in $\M_b(\bfO\setminus\bfC^r)$ for all but at most countably many $r>0$. Hence, for such $r$,
$p_r(\mu^{(r)}_n,\mu^{(r)})[1+p_r(\mu^{(r)}_n,\mu^{(r)})]^{-1}\to 0$, which by the dominated convergence theorem implies that $d_{\MO}(\mu_n,\mu)\to 0$.

(ii) Separability: For $r>0$ let $D_r$ be a countable dense set in $\M_b(\bfS\setminus\bfC^r)$ with the weak topology. Let $D$ be the union of $D_r$ for rational $r > 0$. Then $D$ is countable. Let us show $D$ is dense in $\MO$. Given $\vep > 0$ and $\mu \in \MO$ pick $r'>0$ such that $\int_0^{r'} e^{-r}\ud r < \vep / 2$. Take $\mu_{r'} \in D_{r'}$ such that $p_{r'}(\mu_{r'}, \mu^{(r')}) < \vep/2$. Then $p_r(\mu_{r'}^{(r)}, \mu^{(r)}) < \vep/2$ for all $r > r'$. In particular, $d_{\MO}(\mu_{r'}, \mu) < \vep$. 

Completeness: Let $\{\mu_n\}$ be a Cauchy sequence for $d_{\MO}$. Then $\{\mu_n^{(r)}\}$ is a Cauchy sequence for $p_r$ for all but at most countably many $r > 0$. Since $\bfS$ is separable and complete, its closed subspace $\bfS\setminus\bfC^r$ is separable and complete. Therefore, 
$\M_b(\bfS\setminus\bfC^r)$ is complete which implies that $\{\mu_n^{(r)}\}$ has a limit $\mu_r$. These limits are consistent in the sense that $\mu_{r'}^{(r)} = \mu_{r}$ for $r' < r$.  On $\scrS_{\bfO}$ set $\mu(A) = \lim_{r \to 0} \mu_r(A \cap \bfS\setminus\bfC^r)$. Then $\mu$ is a measure. Clearly, $\mu \geq 0$ and $\mu(\es) = 0$. Moreover, $\mu$ is countably additive: for disjoint $A_n\in \scrS_{\bfO}$ the monotone convergence theorem implies that 
\begin{align*}
  \mu(\cup_n A_n) 
  &=\lim_{r\to 0}\mu_r(\cup_nA_n\cap[\bfS\setminus\bfC^r])\\ 
  &=\lim_{r \to 0} \sum_n \mu_r(A_n \cap [\bfS\setminus\bfC^r]) 
  =\sum_n \mu(A_n).
\end{align*}
\qed 

\subsubsection{Proof of Theorem \ref{mapthm}.}\label{subsubsec:mapthm}
Firstly,  $D_h \in \scrS_{\bfO}$ \cite[p.~243]{billingsley:1999}.
Take $A' \in \scrS_{\bfO'}$ bounded away from $\bfC'$ with $\mu {\circ }h^{-1}(\partial A')=0$.
Since $\partial h^{-1}(A') \subset h^{-1}(\partial A') \cup D_h$
(see e.g.~(A2.3.2) in \cite{daley:vere-jones:1988}), we have
$\mu(\partial h^{-1}(A'))\leq \mu h^{-1}(\partial A') + \mu(D_h)=0$.
Since $\mu_n\to\mu$ in $\MO$, $\mu(\partial h^{-1}(A')) = 0$, and $h^{-1}(A')$ is bounded away from $\bfC$, it follows from Theorem \ref{portthm} (iv) that 
$\mu_n h^{-1}(A)\to\mu h^{-1}(A)$. Hence, $\mu_nh^{-1}\to\mu h^{-1}$ in $\M_{\bfO'}$. 
\qed

\subsubsection{Proof of Corollary \ref{cor:variant1}.}
Take $A'\subset \bfS'\setminus \bfC'$ such that $d'(A',\bfC')>0$. 
We claim this implies $d(h^{-1}(A'), \bfC)>0$. Otherwise, if $d(h^{-1}(A'),\bfC)=0$, there exist $x_n \in h^{-1}(A')$ and $y_n \in \bfC$ such that $d(x_n,y_n)\to 0.$ Then $h(x_n)\in A',\,h(y_n) \in h(\bfC)=\bfC' $ and if $h$ is uniformly continuous, then $d'(h(x_n),h(y_n))\to 0$ so that $d'(A', \bfC')=0$, a
 contradiction.
\qed

\subsubsection{Proof of Corollary \ref{cor:2ndVariant}.}\label{subsubsec:2ndVariant}
The proof of Corollary \ref{cor:variant1} shows that it suffices if either $\{x_n\}$ or $\{y_n\}$ has a limit point. In the former case, if $x_{n'} \to x$ for some subsequence $n' \to \infty$, then $ d(x,y_{n'})\to 0 $ and $y_{n'}\to x  \in \bfC $ and $h(y_{n'})\to h(x) $ so {$d'(A',\bfC')=0$} again
  giving a contradiction. Note if $\bfS$ is compact than $\{x_n\}$ has  a limit point. On the other hand, if $\{y_n\}$ has a limit point   then there exists an infinite subsequence $\{n'\} $ and  $y_{n'}\to y\in C$ so that $d(x_{n'}, y)\to 0$. Thus if $h$ is continuous, $h(x_{n'}) \to  h(y) \in h(\bfC)=\bfC'$ which contradicts {$d'(A',\bfC')>0$}.  Note if $\bfC $ is compact, then $\{y_n\}$ has a limit point and in particular if $\bfC=\{s_0\}$. Thus we have the second variant.
\qed

\subsubsection{Proof of Theorem \ref{rc0}.}
\label{subsubsec:rc0}
Suppose $M \subset \MO$ is relatively compact. Let
$\{\mu_n\}$ be a subsequence in $M$. Then there exists a
convergent subsequence $\mu_{n_k} \to \mu$ for some $\mu \in M^-$. By Theorem \ref{portthm} (v), there exists a sequence
$\{r_i\}$ with $r_i \downarrow 0$ such that
$\mu_{n_k}^{(r_i)}\to\mu^{(r_i)}$ in
$\M_b(\bfS \setminus\bfC^{r_i})$. Hence, $M^{(r_i)}$ is
relatively compact in $\M_b(\bfS\setminus\bfC^{r_i})$ for each such
$r_i$.    

Conversely, suppose there exists a sequence $\{r_i\}$ with $r_i\downarrow 0$ such that $M^{(r_i)} \subset \M_b(\bfS\setminus\bfC^{r_i})$ is relatively compact for each $i$, and let $\{\mu_n\}$ be a sequence of elements in $M$. We use a diagonal argument to find a convergent subsequence. 
Since $M^{(r_1)}$ is relatively compact there exists a subsequence $\{\mu_{n_1(k)}\}$ of $\{\mu_n\}$ such that $\mu_{n_1(k)}^{(r_1)}$ converges to some $\mu_{r_1}$ in $\M_b(\bfS\setminus\bfC^{r_1})$. Similarly since $M^{(r_2)}$ is relatively compact and $\{\mu_{n_1(k)}\} \subset M$ there exists a subsequence $\{\mu_{n_2(k)}\}$ of $\{\mu_{n_1(k)}\}$ such that $\mu_{n_2(k)}^{(r_2)}$ converges to some $\mu_{r_2}$ in $\M_b(\bfS\setminus\bfC^{r_2})$. Continuing like this; for each $i \geq 3$ let $n_i(k)$ be a subsequence of $n_{i-1}(k)$ such that $\mu_{n_i(k)}^{(r_i)}$ converges to some $\mu_{r_i}$ in $\M_b(\bfS\setminus\bfC^{r_i})$. Then the diagonal sequence $\{\mu_{n_k(k)}\}$ satisfies $\mu_{n_k(k)}^{(r_i)}\to \mu_{r_i}$ in $\M_b(\bfS\setminus\bfC^{r_i})$ for each $i \geq 1$. 
Take $f \in \C_{\bfO}$. There exists some $i_0 \geq 1$ such that  $f$ vanishes on $\bfS\setminus\bfC^{r_i}$ for each $i \geq i_0$. In particular $f \in \C_b(\bfS\setminus\bfC^{r_i})$ for each  $i \geq i_0$ and 
\begin{align*}
  \int f \ud \mu_{r_i} = \lim_{k} \int f \ud \mu_{n_k(k)}^{(r_i)} =
  \lim_{k} \int f \ud \mu_{n_k(k)}^{(r_{i_0})} = \int f \ud \mu_{r_{i_0}}.
\end{align*}
Hence, we can define $\mu': \C_{\bfO}\to [0,\infty]$ by $\mu'(f)=\lim_{i \to\infty}\int f\ud\mu_{r_i}$. This $\mu'$ induces a measure $\mu$ in $\MO$. Indeed, for $A \in \scrS_{\bfO}$ we can find a sequence $f_n\in\C_{\bfO}$ such that $0 \leq f_n \uparrow I_A$ and set $\mu(A)=\lim_n\mu'(f_n)$. If $A \in\bfS\setminus\bfC^{r}$ for some $r> 0$, then there exists $f_n\in\C_{\bfO}$ such that $f_n \downarrow I_A$ and hence $\mu(A) \leq \mu'(f_n) < \infty$. Thus, $\mu$ is finite on sets $A\in\bfS\setminus\bfC^{r}$ for some $r>0$. To show that $\mu$ is countably additive, let $A_1, A_2, \dots$ be disjoint sets in $\scrS_{\bfO}$ and $0 \leq f_{nk} \uparrow I_{A_k}$ for each $k$. Then $\sum_k f_{nk}\uparrow I_{\cup_k A_k}$ and, by Fubini's theorem and the monotone convergence theorem, it holds that
\begin{align*}
  \mu(\cup_k A_k) = \lim_n \mu'\Big(\sum_k f_{nk}\Big) = \sum_k  \lim_n
  \mu'(f_{nk}) = \sum_k \mu(A_k).
\end{align*}
By construction $\int f \ud \mu = \mu'(f)$ for each $f \in\C_{\bfO}$. Hence, $\int f\ud\mu_{n_k(k)}\to \int f \ud \mu$ for each $f \in \C_{\bfO}$, and we conclude that $M$ is relatively compact in $\MO$. 
\qed

\subsubsection{Proof of Theorem \ref{rc}.}\label{subsubsec:rc}
Suppose $M \subset \MO$ is relatively compact. By Theorem \ref{rc0},
there exists a sequence $\{r_i\}$ with $r_i \downarrow 0$ such that
$M^{(r_i)} \subset \M_b(\bfS \setminus \bfC^{r_i})$ is relatively compact
for each $r_i$. Prohorov's theorem (Theorem A2.4.1 in \cite{daley:vere-jones:1988})
implies that \eqref{rc1} and \eqref{rc2} hold. 

Conversely, suppose there exists a sequence $\{r_i\}$ with $r_i \downarrow 0$ 
such that \eqref{rc1} and \eqref{rc2} hold. Then, by
Prohorov's theorem, $M^{(r_i)} \subset \M_b(\bfS \setminus \bfC^{r_i})$ is
relatively compact for each $i$. By Theorem \ref{rc0}, 
$M \subset \MO$ is relatively compact.
\qed

\section{Regularly varying sequences of measures}\label{secrv}
\subsection{Scaling}\label{subsec:scaling}
The usual notion of regular variation involves comparisons along a ray
and requires a concept of scaling or multiplication.  We approach the
scaling idea in a general complete, separable metric space
$\mathbb{S}$ by postulating what is required for a pleasing theory.
Given any real number $\lambda>0$  and any $x\in \bfS$, we assume
there exists a mapping
$(\lambda,x)\mapsto \lambda x$ from $(0,\infty) \times \bfS$ into
$\bfS$ 
satisfying:
\begin{itemize}
\item[(A1)]\label{A1} the mapping $(\lambda,x)\mapsto \lambda x$ is continuous,
\item[(A2)]\label{A2} $1x=x$ and $\lambda_1(\lambda_2 x)=(\lambda_1\lambda_2)x$.
\end{itemize}
Assumptions (A1) and (A2) allow definition of a cone $\bfC \subset
\mathbb{S}$ as a
 set satisfying $x\in \bfC$ implies $\lambda x \in \bfC$ for any
$\lambda>0$. 
For this section, fix a closed cone $\bfC \subset \bfS$
and then $\bfO:=\bfS\setminus \bfC$ is 
 an open cone. We require that 
\begin{itemize}
\item[(A3)]\label{A3} $d(x,\bfC) < d(\lambda x,\bfC)$ if $\lambda>1$ and $x\in\bfO$.
\end{itemize}

\subsubsection{Examples to fix ideas:}\label{subsubsec:eg3} To
emphasize the flexibility allowed by our assumptions, consider the
following  circumstances {all of which satisfy (A1)--(A3).}
\begin{enumerate}
\item\label{4.1}
Let $\bfS=\R^2$ and $\bfC=(\{0\}\times\R)\cup(\R\times\{0\})$ and for
$\gamma_1>0,\, \gamma_2>0$
define $(\lambda,(x_1,x_2))\mapsto (\lambda^{1/\gamma_1}x_1,\lambda^{1/\gamma_2}x_2)$.

\item\label{4.2}
Set $\bfS=\R^2$ and  $\bfC=\R\times\{0\}$.
Define $(\lambda,(x_1,x_2))\mapsto (x_1,\lambda x_2)$.

\item\label{item:polar} Set $\bfS=[0,\infty)\times\{x\in\R_+^2:\|x\|=1\}$ 
and $\bfC =\{0\}\times \{x\in\R_+^2:\|x\|=1\}$. For $\lambda>0$, 
define $(\lambda, (r,a))\mapsto (\lambda r,a).$
\end{enumerate}

\subsection{Regular variation}\label{sec:regvar}
Recall from e.g.~\cite{bingham:goldie:teugels:1987} that a positive
measurable function $c$ defined on $(0,\infty)$ is regularly varying
with index $\rho \in \R$ if $\lim_{t\to\infty}c(\lambda
t)/c(t)=\lambda^{\rho}$ for all $\lambda > 0$. Similarly, a sequence
$\{c_n\}_{n\geq 1}$ of positive numbers is regularly varying with
index $\rho \in \R$ if $\lim_{n\to\infty}c_{[\lambda
  n]}/c_n=\lambda^{\rho}$ for all $\lambda > 0$.  Here $[\lambda n]$
denotes the integer part of $\lambda n$.  

\begin{dfn}\label{rvdefseq}
A sequence $\{\nu_n\}_{n \geq 1}$ in  $\MO$ is regularly varying 
if there exists an increasing sequence $\{c_n\}_{n \geq 1}$ of positive numbers which is regularly varying and a nonzero $\mu \in \MO$ such that $c_n\nu_n \to \mu$ in $\MO$ as $n\to\infty$.
\end{dfn}

The choice of terminology is motivated by the fact that
$\{\nu_n(A)\}_{n\geq 1}$ is a regularly varying sequence for each set
$A \in \scrS_{\bfO}$ bounded away from $\bfC$, $\mu(\partial A)=0$ and
$\mu(A)>0$.  We will now define regular variation for a single measure
in $\MO$.  

\begin{dfn}\label{rvdefmeas}
A measure $\nu\in\MO$ is regularly varying if the sequence $\{\nu(n\cdot)\}_{n\geq 1}$ in $\MO$ is regularly varying.
\end{dfn}

There are many equivalent formulations of regular variation
for a measure $\nu \in \MO$. Some are  natural for statistical
inference. Consider the following statements.  

\begin{itemize}
\item[(i)]
  There exist a nonzero $\mu \in \MO$ and a regularly varying sequence $\{c_n\}_{n\geq 1}$ of positive numbers such that $c_n\nu(n\cdot)\to\mu(\cdot)$ in $\MO$ as $n\to\infty$.
\item[(ii)] 
  There exist a nonzero $\mu \in \MO$ and a regularly varying function $c$ such that 
  $c(t)\nu(t\cdot)\to\mu(\cdot)$ in $\MO$ as $t\to\infty$.
\item[(iii)]
  There exist a nonzero $\mu \in \MO$ and a set $E \in \scrS_{\bfO}$ bounded away from $\bfC$  such that $\nu(tE)^{-1}\nu(t\cdot)\to\mu(\cdot)$ in $\MO$ as $t\to\infty$.
\item[(iv)]
  There exist a nonzero $\mu \in \MO$ and an increasing sequence $\{b_n\}_{n\geq 1}$ of positive numbers such that $n\nu(b_n\cdot)\to\mu(\cdot)$ in $\MO$ as $n\to\infty$.
  \item[(v)]
  There exist a nonzero $\mu \in \MO$ and an increasing function $b$ of such that $t\nu(b(t)\cdot)\to\mu(\cdot)$ in $\MO$ as $t\to\infty$.
\end{itemize}

\begin{thm}\label{thmrvnt}
The statements (i)-(v) are equivalent and each statement implies that the limit measure $\mu$ has the homogeneity property
\begin{equation}\label{eqn:scaling}
\mu(\lambda A)=\lambda^{-\alpha}\mu(A)
\end{equation}
for some $\alpha\geq 0$ and all $A \in \scrS_{\bfO}$ and $\lambda > 0$.
\end{thm}

Notice that a regularly varying measure does not correspond to a
single scaling parameter $\alpha$
unless the multiplication operation with scalars is
fixed.

\subsection{More  examples}\label{subsec:egsMore}
We amplify the discussion of Section \ref{subsubsec:eg3}.

\subsubsection{Continuation of Section \ref{subsubsec:eg3}.}\label{subsubsec:continue}

\begin{exm}\label{eg:3_1}
Consider again the context of Section \ref{subsubsec:eg3}, item
\ref{4.1} where $\bfS=\R^2$ and let $\bfC=(\{0\}\times\R)\cup(\R\times\{0\})$.
Consider two independent Pareto random variables: Let $X_1$ be
Pa$(\gamma_1)$ and $X_2$ be Pa$(\gamma_2)$.
Define $(\lambda,(x_1,x_2))\mapsto (\lambda^{1/\gamma_1}x_1,\lambda^{1/\gamma_2}x_2)$.
For $a,b>0$
\begin{align*} 
t^2 P\bigl[t^{-1} (X_1,X_2)\in (a,\infty)\times (b,\infty)\bigr]
&=t\Prob \bigl[X_1>t^{1/\gamma_1}a]t\Prob[X_2>t^{1/\gamma_2}b \bigr]\\
&=a^{-\gamma_1}b^{-\gamma_2}.
\end{align*} 
According to our definition, the distribution of $(X_1,X_2)$ is
regularly varying on $\mathbb{S} \setminus \bfC$.
The limit measure therefore has the scaling property: For $\lambda >0$,
\begin{align*}
\mu \bigl(\lambda[(a,\infty)\times (b,\infty)] \bigr)&=\mu
\bigl((\lambda^{1/\gamma_1}a,\infty)\times
(\lambda^{1/\gamma_1}b,\infty)\bigr)\\
&=\lambda^{-2}a^{-\gamma_1}b^{-\gamma_2}
=\lambda^{-2}\mu((a,\infty)\times (b,\infty)).
\end{align*}
\end{exm}

\begin{exm}\label{eg:3_2}
Recall Section \ref{subsubsec:eg3}, item \ref{4.2} where
 $\bfS=\R^2$ and  $\bfC=\R\times\{0\}$ with  $(\lambda,(x_1,x_2))\mapsto (x_1,\lambda x_2)$.
Suppose $X_1,X_2$ are independent with  $X_1$ being N$(0,1)$ and
$X_2$ being  Pa$(\gamma)$. For $a,b>0$
\begin{align*}
t^{\gamma}\Prob\bigl[ t^{-1}(X_1,X_2)\in (a,\infty)\times (b,\infty)\bigr]
=\Prob [X_1>a] t^{\gamma}\Prob[X_2>tb]
=(1-\Phi(a))b^{-\gamma}, 
\end{align*} 
implying that the distribution of $(X_1,X_2)$ is regularly varying.
For $\lambda >0$, the limit measure  has the scaling property,
\begin{align*}
\mu(\lambda[(a,\infty) &\times (b,\infty)])=\mu((a,\infty)\times (\lambda b,\infty))\\
&=\lambda^{-\gamma}(1-\Phi(a))b^{-\gamma}
=\lambda^{-\gamma}\mu((a,\infty)\times (b,\infty)).
\end{align*}
\end{exm}

\subsubsection{Examples using the mapping Theorem \ref{mapthm}.}\label{subsubsec:applyMapping}

\begin{exm}[Cf. \cite{maulik:resnick:rootzen:2002}, Theorem 2.1, page 677]
\label{ex:mappingthm}
Suppose $\bfS=[0,\infty)^2,\, \bfC=[0,\infty)\times \{0\}$ so that 
$$\mathbb{O}=\bfS\setminus \bfC =[0,\infty)^2 \setminus
[0,\infty)\times \{0\}=[0,\infty)\times (0,\infty)=:\bfD_\sqcap.$$
Define $h:\bfD_\sqcap\mapsto \bfD_\sqcap$ by $h(x,y)=(xy,y).$ If
$A'\subset \bfD_\sqcap$ and $d(A',[0,\infty)\times \{0\})>0, $
then $\inf\{y: (x,y)\in A', \text{ for some }x\}>0$ and so
$h^{-1}(A')=\{(x,y): h(x,y)\in A'\}=\{(x,y): (xy,y)\in A'\}$ is also
at positive distance from $[0,\infty)\times \{0\}.$ So the hypotheses
of Theorem \ref{mapthm} are satisfied with $h$ and if 
${\mu_t \to \mu}$ in $\mathbb{M}_{\mathbb{D}_\sqcap}$
then it follows that ${\mu_t \circ h^{-1} \to \mu\circ h^{-1}}$ 
in $\mathbb{M}_{\mathbb{D}_\sqcap}$. In particular,
suppose for a random vector $(X,Y)$ and scaling function $b(t) \to \infty$,
\begin{align}\label{eqn:regVarD_sqcap}
tP\Bigl[\Bigl(X,\frac{Y}{b(t)} \Bigr) \in \cdot \Bigr] \to {\mu(\cdot)}
\qquad \text{ in $\mathbb{M}_{\mathbb{D}_\sqcap}$}.
\end{align}
This is regular variation of the
distribution of $(X,Y)$ on ${\mathbb{D}_\sqcap}$ with the scaling
function defined as $(\lambda, (x,y))\mapsto (x,\lambda y)$.
The mapping Theorem \ref{mapthm} gives
\begin{align}\label{eqn:regVar2D_sqcap}
tP\Bigl[\Bigl(\frac{XY}{b(t)},\frac{Y}{b(t)} \Bigr) \in \cdot \Bigr]
\to \mu' (\cdot)
\qquad \text{ in $\mathbb{M}_{\mathbb{D}_\sqcap}$}.
\end{align}
where $\mu'=\mu \circ h^{-1}$, which is regular variation with
respect to the traditional scaling $(\lambda, (x,y))\mapsto (\lambda x,\lambda y).$

Conversely, define
$g:\bfD_\sqcap\mapsto \bfD_\sqcap$ by $g(x,y)=(x/y,y).$ One observes
$g$ is continuous and obeys the bounded away condition and so 
${\mu_t \to \mu}$ in $\mathbb{M}_{\mathbb{D}_\sqcap}$ implies
${\mu_t \circ h^{-1} \to \mu\circ h^{-1}}$ in $\mathbb{M}_{\mathbb{D}_\sqcap}$.

The summary is that \eqref{eqn:regVarD_sqcap} and
\eqref{eqn:regVar2D_sqcap} are equivalent.
\end{exm}

\begin{exm}[Polar coordinates]\label{eg:polar}
Set 
\begin{align*}
&\bfS=[0,\infty)^2,\;\bfC=\{0\},\;
\bfO=[0,\infty)^2\setminus \{0\}\\
\intertext{with scaling function $(\lambda, x)=(\lambda,
  (x_1,x_2))\mapsto \lambda x=  (\lambda x_1,\lambda x_2).$ For
some choice of norm $x\mapsto \|x\|$ define $\aleph=\{x \in
\bfS: \|x\|=1\}.$ Also define}
&\bfS'=[0,\infty) \times \aleph,\;\bfC'=\{0\} \times \aleph,\;
\bfO'=(0,\infty)\times \aleph,
\end{align*}
and scaling operation on $\bfO'$ is $\bigl(\lambda,
(r,a)\bigr)\mapsto (\lambda r,a).$ The map
$$h(x)=\bigl(\|x\|, x/\|x\|\bigr)$$
from $\bfO \mapsto \bfO' $ is continuous. Let
$d$ and $d'$ be the 
 the metrics on
$\bfS$ and  $\bfS'$. 
Suppose $X$ has a regularly varying distribution on $\bfO$ so that
for some $b(t)\to\infty$,
\begin{equation}\label{eqn:regVarCartesian}
tP\bigl[X/b(t) \in \cdot \,\bigr] \to \mu(\cdot)
\end{equation}
in $\MO$ for some limit measure $\mu$. We show
$h(X)=:(R,\Theta)$ has a regularly varying 
distribution on $\bfO'$. We
apply Theorem \ref{mapthm} so suppose $A'\subset \bfO' $  satisfies
$d'(A', \{0\}\times \aleph)>0,$ that is, $A'$ is bounded away from the
deleted portion of $\bfS'$. Then
$\inf\{r>0: (r,a) \in A'\}=\delta>0$ and $h^{-1}(A')=\{x \in \bfO:
\bigl( \|x\|, x/\|x\|\bigr) \in A'\}$ satisfies $\inf\{\|x\|:x \in h^{-1}(A')\}=\delta'>0$. So the hypotheses of Theorem \ref{mapthm} are satisfied and allow the conclusion that 
\begin{align}\label{eqn:regVarPolar}
 tP\bigl[ \bigl(\frac{R}{b(t)},\Theta \bigr) \in \cdot \,\bigr] 
\to {\mu \circ h^{-1}(\cdot)} \qquad \text{ in }\mathbb{M}_{\bfO'}.
\end{align}

Conversely, given regular variation on $\bfO'$ as in
\eqref{eqn:regVarPolar}, define
$g: \bfO'\mapsto \bfO$ by $g(r,a)=r a$. Mimic the verification
above to conclude \eqref{eqn:regVarPolar} implies \eqref{eqn:regVarCartesian}.
\end{exm}

\begin{exm} {Examples \ref{ex:mappingthm} and \ref{eg:polar}
    typify the following paradigm.}
Consider two triples $(\bfS,\bfC,\bfO)$ and $(\bfS',\bfC',\bfO')$, and
a {homeomorphism }
 $h:\bfO\to\bfO'$ with the property that $h^{-1}(A')$ is bounded away from $\bfC$ if $A'$ is bounded away from $\bfC'$.
The multiplication by a scalar $(\lambda,x)\mapsto \lambda x$ in $\bfO$ gives rise to the multiplication by a scalar $(\lambda',x')\mapsto \lambda' x':=h(\lambda' h^{-1}(x'))$ in $\bfO'$. Notice that $(\lambda',x')\mapsto \lambda' x'$ is continuous, $1x'=x'$, and $\lambda_1'(\lambda_2'(x'))=(\lambda_1'\lambda_2')x'$. We also need to check that $d'(\lambda'x',\bfC')>d'(x',\bfC')$ if $\lambda'>1$.
\end{exm}

\subsection{Proofs}
\subsubsection{Preliminaries.}\label{subsubsec:prelim3}

For $A\in\scrS_{\bfO}$, write $S(A)=\{\lambda x: x\in A, \lambda \geq 1\}$. 

\begin{lem}\label{lem:goodsets} 
Let $\mu \in \MO$ be nonzero. There exists $x\in\bfO$ and $\delta>0$ such that $S(B_{x,\delta})$ is bounded away from $\bfC$, $\mu(S(B_{x,\delta}))>0$, and $\mu(\partial rS(B_{x,\delta}))=0$ for $r\geq 1$ in some set of positive measure containing $1$.
\end{lem}
\begin{proof} 
The first two properties are obvious. In order to prove the final claim, set
$\gamma(r)=d(\partial rS(B_{x,\delta}),\bfC)$. Notice that $\gamma(r)=d(\partial rB_{x,\delta},\bfC)$ and that $\partial rS(B_{x,\delta})\subset \bfO\setminus \bfC^{\gamma(r)}$. Choose $x\in\bfO$ and $\delta>0$ such that $\mu(\partial S(B_{x,\delta}))=0$ and $\mu(\partial (\bfO\setminus \bfC^{\gamma(1)}))=0$, and such that $\gamma(r')>\gamma(1)$ for some $r'>1$. The existence of such $x$ and $\delta$ follows from Lemma \ref{newremii}. Lemma \ref{newremii} also implies that $\mu(\partial (\bfO\setminus\bfC^{\gamma}))=0$ for all but at most countably many $\gamma\in [\gamma(1),\gamma(r')]$. Since $\gamma(r)$ is a nondecreasing and continuous function and $\gamma(r')>\gamma(1)$, there exists a set $R\subset [1,r']$ of positive Lebesgue measure, with $1\in R$, such that $\mu(\partial (\bfO\setminus\bfC^{\gamma(r)}))=0$ for $r\in R$.
\end{proof} 

Given $\mu\in\MO$, let $\scrA_{\mu}$ denote the set of $\mu$-continuity sets $A\in \scrS_{\bfO}$ bounded away from $\bfC$ satisfying $S(A)=A$.

\begin{lem}\label{lem:convdetclassSA}
If $\mu_n(A)\to\mu(A)$ for all $A\in \scrA_{\mu}$, then $\mu_n\to\mu$ in $\MO$.
\end{lem}

\begin{proof}
Let $\scrD_{\mu}$ denote the $\pi$-system of finite differences of sets of the form $A_1\setminus A_2$ for $A_1,A_2\in\scrA_{\mu}$ with $A_2\subset A_1$.
Take $x \in \bfO$ and $\vep>0$ such that $B_{x,\vep}$ is bounded away from $\bfC$. The sets $\partial S(B_{x,r})$, for $r\in (0,\vep)$, are disjoint. Similarly, the sets $\partial B_{x,r}$, for $r\in (0,\vep)$, are disjoint. Therefore, $\mu(\partial S(B_{x,r}))=\mu(\partial B_{x,r})=0$ for all but at most countably many $r\in(0,\vep)$. Moreover, $B_{x,r}=S(B_{x,r})\setminus (S(B_{x,r})\setminus B_{x,r})$, so $B_{x,r}\in \scrD_{\mu}$ for all but at most countably many $r\in (0,\vep)$.
Hence, there exists $A\in \scrD_{\mu}$ such that $x\in A^{\circ}\subset A\subset B_{x,\vep}$.
Moreover, for any $x$ in an open set $G$ bounded away from $\bfC$,
there exists $A\in\scrD_{\mu}$ such that $x\in A^{\circ}\subset
A\subset G$. Since $\bfO$ is separable we find (as in the proof of
Theorem 2.3 in \cite{billingsley:1999}) that there is a countable
subcollection $\{A_{x_i}^{\circ}\}$ of $\{A_x^{\circ}:x\in G\}$,
$A_x^{\circ}\in\scrD_{\mu}$, that covers $G$ and that $G=\cup_i
A_{x_i}^{\circ}$. The inclusion-exclusion argument in the proof of
Theorem 2.2 in \cite{billingsley:1999} this implies that
$\liminf_{n}\mu_n(G)\geq \mu(G)$ for all open sets $G$ bounded away
from $\bfC$.  
Any closed $F$ bounded away from $\bfC$ is a subset of an open $\mu$-continuity set $A=\bfO \setminus C^r$ for some $r>0$. Notice that $A\in \scrA_{\mu}$. Therefore
\begin{align*}
\mu(A)-\limsup_{n}\mu_n(F)=\liminf_{n}\mu_n(A\setminus F)\geq \mu(A\setminus F)=\mu(A)-\mu(F),
\end{align*} 
i.e.~$\limsup_{n}\mu_n(F)\leq \mu(F)$. The conclusion follows from Theorem \ref{portthm}.
\end{proof}

\subsubsection{Proof of Theorem \ref{thmrvnt}.}
The proof is structured as follows. We first prove that (iii) implies the homogeneity property in \eqref{eqn:scaling} of the limit measure $\mu$. Then we prove that the statements (i)-(v) are equivalent and that the limit measures are the same up to a constant factor.

Suppose that (iii) holds and take $E'=S(B_{x,\delta})$ satisfying the conditions in Lemma \ref{lem:goodsets}. Then, for $\lambda\geq 1$ in a set of positive measure containing $1$,
\begin{align*}
  \frac{\nu(t\lambda E')}{\nu(tE')}
  = \frac{\nu(t\lambda E')}{\nu(tE)}\frac{\nu(tE)}{\nu(tE')}
  \to \frac{\mu(\lambda E')}{\mu(E')} \in (0,\infty)
\end{align*}
as $t\to\infty$. It follows from Theorem 1.4.1 in \cite{bingham:goldie:teugels:1987} that $t\mapsto \nu(tE')$ is regularly varying and that $\mu(\lambda E')=\lambda^{-\alpha}\mu(E')$ some $\alpha\in\R$ and all $\lambda>0$. Property (A3) implies that $\nu(t\lambda E')/\nu(t E')\leq 1$ for $\lambda\geq 1$ so $\alpha \geq 0$. Moreover, $\mu(\partial (\lambda E'))=0$ for all $\lambda > 0$ and $\nu(tE')^{-1}\nu(t\cdot)\to\mu(E')^{-1}\mu(\cdot)$ in $\MO$ as $t\to\infty$.
In particular, if $A \in \scrA_\mu$, then for any $\lambda > 0$, 
\begin{align*}
  \frac{\nu(t\lambda A)}{\nu(tE')}
  = \frac{\nu(t\lambda A)}{\nu(t\lambda E')}
  \frac{\nu(t\lambda E')}{\nu(tE')}
  \to\lambda^{-\alpha}\frac{\mu(A)}{\mu(E')}
\end{align*}
as $t\to\infty$. Hence, for $A \in \scrA_\mu$ $\mu(\lambda A)=\lambda^{-\alpha}\mu(A)$ for all $\lambda > 0$. By Lemma \ref{lem:convdetclassSA} it follows that  
$\mu(\lambda A)=\lambda^{-\alpha}\mu(A)$ for all $A \in \scrS_{\bfO}$ and $\lambda>0$.

Suppose that (i) holds and set $c(t)=c_{[t]}$.
For each $A \in \scrA_{\mu}$ and $t\geq 1$ it holds that
\begin{align}\label{ntbounds}
  \frac{c_{[t]}}{c_{[t]+1}}c_{[t]+1}\nu(([t]+1)A)
  \leq c(t)\nu(tA)\leq c_{[t]}\nu([t]A).
\end{align}
Since $\{c_n\}_{n\geq 1}$ is regularly varying it holds that $\lim_{n\to\infty}c_n/c_{n+1}=1$.
Hence, $\lim_{t\to\infty}c(t)\nu(tA)=\mu(A)$ for all $A\in\scrA_\mu$. It follows from Lemma \ref{lem:convdetclassSA} that (ii) holds.

Suppose that (ii) holds. Then $c_{[t]}\nu([t]\cdot)\to \mu(\cdot)$ in $\MO$. Moreover, $\{c_{[t]}\}$ is a regularly varying sequence since $c(t)$ is a regularly varying function. Therefore, (ii) implies (i).
 
Suppose that (ii) holds. Take a set $E\in\scrS_{\bfO}$ bounded away from $\bfC$ such that $\nu(tE),\mu(E)>0$ and $\mu(\partial E)=0$. Then 
\begin{align*}
  \frac{\nu(t\cdot)}{\nu(tE)}
  =\frac{c(t)\nu(t\cdot)}{c(t)\nu(tE)}
  \to \frac{\mu(\cdot)}{\mu(E)}
\end{align*}
as $t\to\infty$. Hence, by Theorem \ref{portthm} (ii), (iii) holds.

Suppose that (iii) holds. It was already proved in (a) above that statement (iii) implies that $t\mapsto \nu(tE)$ is regularly varying with index $-\alpha\leq 0$. Setting $c(t)=1/\nu(tE)$ implies that $c(t)$ is regularly varying with index $\alpha$ and that $c(t)\nu(t\cdot)\to\mu(\cdot)$ in $\MO$. This proves that (iii) implies (ii). Up to this point we have proved that statements (i)-(iii) are equivalent.

Suppose that (iv) holds. Set $b(t)=b_{[t]}$ and take $A\in\scrA_{\mu}$. Then 
\begin{align*}
\frac{[t]}{[t+1]}[t+1]\nu(b_{[t+1]}A)\leq t\nu(b(t)A)\leq \frac{[t+1]}{[t]}[t]\nu(b_{[t]}A)
\end{align*}
from which it follows that $\lim_{t\to\infty}t\nu(b(t)A)=\mu(A)$. It follows from Lemma \ref{lem:convdetclassSA} that (v) holds. 
If (v) holds, then it follows immediately that also (iv) holds. Hence, statements (iv) and (v) are equivalent.

Suppose that (iv) holds.
Take $E$ such that $\mu(\partial E)=0$ and $\mu(E)>0$.
For $t>b_1$, let $k=k(t)$ be the largest integer with $b_k\leq t$.
Then $b_k \leq t < b_{k+1}$ and $k\to\infty$ as $t\to\infty$.
Hence, for $A\in \scrA_{\mu}$,
\begin{align*}
  \frac{k}{k+1}
  \frac{(k+1)\nu(b_{k+1}A)}{k\nu(b_k E)}
  \leq \frac{\nu(tA)}{\nu(tE)}
  \leq \frac{k+1}{k}
  \frac{k\nu(b_{k}A)}{(k+1)\nu(b_{k+1}E)}
\end{align*}
from which it follows that $\lim_{t\to\infty}\nu(tA)/\nu(tE)=\mu(A)/\mu(E)$. It follows from Lemma \ref{lem:convdetclassSA} that (iii) holds. Hence, each of the statments (iv) and (v) implies each of the statements (i)-(iii).

Suppose that (iii) holds. Then $c(t):=1/\nu(tE)$ is regularly varying at infinity with index $\alpha\geq 0$. If $\alpha>0$, then $c(c^{-1}(t))\sim t$ as $t\to\infty$ by Proposition B.1.9 (10) in \cite{dehaan:ferreira:2006} and therefore
\begin{align*}
\lim_{t\to\infty}t\nu(c^{-1}(t)A)=\lim_{t\to\infty}c(c^{-1}(t))\nu(c^{-1}(t)A)=\mu(A)
\end{align*}
for all $A\in\scrS_{\bfO}$ bounded away from $\bfC$ with $\mu(\partial A)=0$.
If $\alpha=0$, then Proposition 1.3.4 in \cite{bingham:goldie:teugels:1987} says that there exists a continuous and increasing function $\widetilde{c}$ such that $\widetilde{c}(t)\sim c(t)$ as $t\to\infty$. In particular, $\widetilde{c}(\widetilde{c}^{-1}(t))=t$ and 
\begin{align*}
t\nu(\widetilde{c}^{-1}\,\cdot)=\widetilde{c}(\widetilde{c}^{-1}(t))\nu(\widetilde{c}^{-1}\,\cdot)\to\mu(\cdot)
\end{align*}
in $\M_{\bfO}$ as $t\to\infty$.
Hence, $(v)$ holds.
\qed

\section{$\mathbb{R}_+^\infty$ and $\mathbb{R}_+^p$}\label{seqSpace}
This section considers regular variation {for measures on} the metric spaces
$\mathbb{R}_+^\infty$ and $\mathbb{R}_+^p$ for $p\geq 1$ and applies the theory of
Sections \ref{secconv} and \ref{secrv}. We begin in Section
\ref{subsec:pre} with notation and specification of metrics and then
address
in Section \ref{subsub:cont} continuity properties for the
following maps
\begin{itemize}
\item $\cumsum: (x_1,x_2,\dots)\mapsto (x_1, x_1 + x_2, x_1 + x_2
  +x_3,\dots)$.
\item $\projp:  (x_1,x_2,\dots)\mapsto (x_1,\dots,x_p)$.
\item $\polar: (x_1,\dots,x_p) \mapsto \bigl(\|x\|,  (x_1,\dots,x_p)
  /\|x\|\bigr)$, where $x=(x_1,\dots,x_p)$ and the norm is Euclidean norm on $\R_+^p$. We
  also define the generalized polar coordinate
  transformation in \eqref{eq:gpolarDef} which is necessary for
  estimating the tail measure of regular variation when the Euclidean
  unit sphere $\{x \in \R_+^p: \|x\|=1\}$ is not bounded away from
  $\bfC$ in the space $\R^\infty_+ \setminus \bfC$.
\end{itemize}
Section \ref{subsec:reduce} { reduces the convergence question to finite
dimensions by giving} criteria for reduction of
convergence of measures in $\bfM(\R_+^\infty \setminus \bfC)$
to convergence of projected measures in $\bfM(\R_+^p\setminus
\projp(\bfC))$. {Section \ref{subsub:compareCompact} returns }to
the comparison of vague convergence  with $\bfM$-convergence initiated
in Section \ref{subsec:MvsVague} {with the goal of using} existing results
based on  regular variation in a compactified version $\R_+^p\setminus \{0\}$ in our present context, rather than proving things from
scratch. Section \ref{seqSpace} concludes with Section
\ref{subsec:regVar}, a discussion of regular variation of measures on
$\R^\infty_+ \setminus \bfC$ giving particular attention to hidden
regular variation properties of the distribution of
$X=(X_1,X_2,\dots )$, a sequence of iid non-negative random
variables whose marginal distriutions have regularly varying
tails. This discussion extends naturally to hidden regular variation
properties of an infinite sequence of non-negative decreasing Poisson
points whose mean measure has a regularly varying tail. 
Results for the Poisson sequence provide the basis of 
our approach in the next Section \ref{sec:hrvLevy}
to regular variation of the distribution of a L\'evy
process whose L\'evy measure is regularly varying.

\subsection{Preliminaries.}\label{subsec:pre}
We write $x \in \R_+^\infty$ as $x =(x_1,x_2, \dots )$. For
$p\geq 1$, the projection into 
$\Rplus^p$ is written as $\projp(x) =x_{|p} =(x_1,\dots,x_p)$. To avoid
confusion, we sometimes write $0_\infty=(0,0,\dots ) \in
\Rplus^\infty$ and $0_p  =(0,\dots,0),$  the vector of $p$
zeros. We also augment a vector in $\Rplus^p$ to get a sequence in
$\Rplus^\infty$ and write, for instance,
$$(x_1,\dots,x_p,0_\infty)=(x_1,\dots,x_p,0,0,\dots ).$$

\subsubsection{Metrics.}\label{subsub:metrics} 
{All metrics are equivalent on $\Rplus^p$.}  {The usual metric
  on} $\Rplus^\infty$  is 
$$\dinfty(x,y)=\sum_{i=1}^\infty \frac{|x_i-y_i|\wedge 1}{2^i},$$
and we also  need 
$$\dinfty'(x,y)=\sum_{p=1}^\infty \frac{\bigl( \sum_{l=1}^p
  |x_l-y_l|\bigr)\wedge 1}{2^p}
=\sum_{p=1}^\infty \frac{\|   x_{|p} -y_{|p} \|_1 \wedge 1}{2^p},$$
where $\|\cdot\|_1 $ is the usual $L_1$ norm on $\Rplus^p$.

\begin{prop}\label{prop:equivMetrics}
The metrics $\dinfty$ and $\dinfty'$ are equivalent on $\Rplus^\infty$ and 
$$\dinfty (x,y) \leq \dinfty'(x,y) \leq 2\dinfty (x,y).$$
\end{prop}

\begin{proof} First of all,
\begin{align*}
\dinfty'(x,y) =\sum_{i=1}^\infty \frac{\bigl( \sum_{l=1}^i
  |x_l-y_l|\bigr)\wedge 1}{2^i} \geq 
\sum_{i=1}^\infty \frac{
  |x_i-y_i|\wedge 1}{2^i} =\dinfty (x,y).
\end{align*}
For the other inequality, observe
\begin{align*}
\dinfty '(x,y) =&\sum_{i=1}^\infty \frac{\bigl( \sum_{l=1}^i
  |x_l-y_l|\bigr)\wedge 1}{2^i}
\leq \sum_{i=1}^\infty \frac{ \sum_{l=1}^i
  \bigl(|x_l-y_l|\wedge 1\bigr)}{2^i}\\
=&\sum_{l=1}^\infty  \sum_{i=l}^\infty 2^{-i} 
  \bigl(|x_l-y_l|\wedge 1\bigr)
=\sum_{l=1}^\infty  2\cdot 2^{-l}
  \bigl(|x_l-y_l|\wedge 1\bigr) = 2\dinfty(x,y).
\end{align*}
\end{proof}

\subsection{Continuity of maps.}\label{subsub:cont}
With a view toward applying Corollary \ref{cor:variant1}, we consider
the continuity of several maps. 

\subsubsection{$\cumsum$.}\label{subsub:cumsumCont}
We begin with the map
$\cumsum : \Rplus^\infty \mapsto \Rplus^\infty$
defined by
$$\cumsum (x)=(x_1,x_1+x_2, x_1+x_2 +x_3,\dots).$$

\begin{prop}\label{prop:cumsumCont}
The map $\cumsum : \Rplus^\infty \mapsto \Rplus^\infty$ is uniformly continuous and, in fact, is Lipshitz in the $\dinfty $ metric.
\end{prop}

\begin{proof}
We write
\begin{align*}
\dinfty\bigl(\cumsum (x),&\cumsum (y)\bigr)
=\sum_{i=1}^\infty \frac{ \bigl | \sum_{l=1}^i x_l- \sum_{l=1}^i y_l
  \bigr |\wedge 1}{2^i}\\
\leq &\sum_{i=1}^\infty \frac{ \bigl(\sum_{l=1}^i   | x_l-  y_l  |
  \bigr)\wedge 1}{2^i}
=\dinfty ' (x,y) \leq 2 \dinfty (x,y).
\end{align*}
\end{proof}

We can now apply Corollary \ref{cor:variant1}.
\begin{cor}\label{cor:cumsumRV}
Let $\bfS=\bfS'=\Rplus^\infty$ and suppose $\bfC $ is closed in 
$\Rplus^\infty $ and $\cumsum (\bfC) $ is closed in 
$\Rplus^\infty .$ If for $n \geq 0$,
${\mu_n} \in \bfM(\Rplus^\infty \setminus \bfC)$ and $\mu_n\to\mu_0$ 
in $\bfM(\Rplus^\infty \setminus \bfC)$, then
${\mu_n \circ \cumsum^{-1}} \to \mu_0 \circ \cumsum^{-1}$ 
in $\bfM(\Rplus^\infty \setminus \cumsum(\bfC))$.
\end{cor}

For example, if $\bfC=\{0_\infty\}$, then $\cumsum (\bfC)=\{0_\infty\}$. 
For additional examples, see \eqref{eq:Fj}.

\subsubsection{PROJECTION}\label{subsub:proj}
For $p\geq 1$, recall $\projp (x)= x_{|p} =(x_1,\dots,x_p)$
from $\Rplus^\infty \mapsto \Rplus^p$.
\begin{prop}\label{prop:projCont} 
$\projp : \Rplus^\infty \mapsto \Rplus^p$ is uniformly continuous.
\end{prop}

\begin{proof}
Let $d_p(x_{|p},y_{|p})=\sum_{i=1}^p |x_i-y_i|$ be the usual $L_1$
metric.  Given $0<\epsilon <1$, we must find $\delta >0$ such that
$\dinfty (x,y)<\delta$ implies $d_p(x_{|p},y_{|p}) <\epsilon.$
We try $\delta=2^{-p}\epsilon.$ Then
\begin{align*}
\delta=2^{-p}\epsilon >&\dinfty (x,y)=
\sum_{i=1}^\infty \frac{|x_i-y_i|\wedge 1}{2^i}
\geq \sum_{i=1}^p \frac{|x_i-y_i|\wedge 1}{2^i}
\\
\geq & {2^{-p}} \sum_{i=1}^p {|x_i-y_i|\wedge 1}.
\end{align*}
Therefore
$\epsilon \geq \sum_{i=1}^p {|x_i-y_i|\wedge 1},$ so that
$\epsilon \geq \sum_{i=1}^p |x_i-y_i| =d_p(x_{|p},y_{|p})$.
\end{proof}

Apply Corollary \ref{cor:variant1}:
\begin{cor}\label{cor:projConv}
Let $\bfS=\bfS'=\Rplus^\infty$ and suppose $\bfC$ is closed in 
$\Rplus^\infty$ and $\projp (\bfC) $ is closed in
$\Rplus^\infty.$ If for $n \geq 0$,
${\mu_n} \in \bfM(\Rplus^\infty \setminus \bfC)$ and $\mu_n \to \mu_0$ 
in  $\bfM(\Rplus^\infty \setminus \bfC)$, then
${\mu_n \circ \projp^{-1}} \to \mu_0 \circ \projp^{-1}$ 
in $\bfM(\Rplus^p \setminus \projp(\bfC))$.
\end{cor}

\subsubsection{POLAR COORDINATE
  TRANSFORMATIONS}\label{subsub:polarCont}
The polar coordinate transformation in $\R_+^p$ is heavily relied
  upon when making inferences about the limit measure of regular
  variation. Transforming from Cartesian to polar coordinates disintegrates
  the transformed limit measure into a product measure, one of whose
  factors concentrates on the unit sphere. This factor is called the
  angular measure. Estimating the angular measure and then
  transforming back to Cartesian coordinates provides the most
  reliable inference technique for tail probability estimation in
  $\R_+^p$ using heavy tail asymptotics. See \cite[pages 173ff,
  313]{resnickbook:2007}.
When removing more than $\{0\}$ from $\R^p_+$, the unit sphere
may no longer be bounded away from what is removed and an alternative
technique we call the {\it generalized polar coordinate
  transformation\/} can be used.

We continue the discussion of Example \ref{eg:polar} that relied on {the}
mapping Theorem \ref{mapthm}. Here we rely on
Corollary \ref{cor:2ndVariant}.
Pick a norm $\|\cdot \|$ on
$\Rplus^p$ and define $\aleph=\{x \in \Rplus^p: \|x\|=1\}$. The conventional polar coordinate
transform $\polar: \Rplus^p\setminus \{0_p\} \mapsto (0,\infty)\times\aleph$ has definition,
\begin{equation}\label{eq:polarDef}
\polar (x)=\Bigl(\|x\|, {x}/{\|x\|} \Bigr).
\end{equation}
Compared with the notation of Corollary \ref{cor:2ndVariant}, we have 
$\bfS=\Rplus^p,$ $ \bfC =\{0_p\} $ which is compact in $\bfS$, $ \bfS'=[0,\infty)\times \aleph,
\bfC'=\{0\}\times \aleph$ which is closed in $\bfC').$ Since $\polar $ is
continuous on the domain, we get from Corollary \ref{cor:2ndVariant}
the following.

\begin{cor}\label{cor:polarConv}
Suppose $\mu_n \to \mu_0$ in $\bfM(\Rplus^p \setminus\{0_p\})$. Then 
$$\mu_n \circ \polar^{-1} \to \mu_0 \circ \polar^{-1} 
$$ in $\bfM( (0,\infty)\times \aleph)$.
\end{cor}

When removing more from the state space than just $\{0_p\}$, the
conventional polar coordinate transform \eqref{eq:polarDef} is not useful
if $\aleph $ is not compact, or at least bounded away from what is
removed. For example, if $\bfS\setminus\bfC=(0,\infty)^p$, $\aleph $ is not compact nor bounded away from the removed axes.  
The following generalization \citep{das:mitra:resnick:2013} sometimes
resolves this, provided \eqref{eq:metricScales} below holds. 

{Temporarilly, we proceed generally and assume $\bfS$ is a complete,
separable metric space and that scalar multiplication 
is defined.  If} $\bfC $ is a cone,  $\theta \bfC=\bfC$ for
$\theta \geq 0$. Suppose further that the metric on $\bfS$ satisfies
\begin{equation}\label{eq:metricScales}
d(\theta x, \theta y)=\theta d(x,y),\quad \theta \geq 0, \,(x,y)\in
\bfS \times \bfS.
\end{equation}
Note \eqref{eq:metricScales} holds for a Banach space where distance
is defined by a norm. (It does not hold for $\R_+^\infty$.) 
If we intend to remove the closed cone $\bfC$, 
set $$\aleph_\bfC =\{s \in \bfS\setminus \bfC: d(s,\bfC)=1\},$$ which
plays the role of the unit sphere and $ \bfC'=\{0\}\times
\aleph_\bfC$ is closed.
Define the generalized transform polar coordinate transformation $$\gpolar : \bfS\setminus
\bfC \mapsto (0,\infty) \times \aleph_{\bfC}
=[0,\infty)\times \aleph_\bfC \setminus \Bigl( \{0\}\times
\aleph_\bfC\Bigr) =
\bfS'\setminus \bfC'$$ 
by
\begin{equation}\label{eq:gpolarDef}
\gpolar (s)=\bigl(d(s,\bfC), s/d(s,\bfC)\bigr), \quad s \in \bfS
\setminus \bfC .\end{equation}
Since $\bfC$ is a cone and $d(\cdot,\cdot)$ has property
\eqref{eq:metricScales}, we have for any $s \in \bfS\setminus \bfC$
that 
$$d\Bigl(\frac{s}{d(s,\bfC)},\bfC\Bigr)=d\Bigl(\frac{s}{d(s,\bfC)},\frac{1}{d(s,\bfC)}\bfC\Bigr)
=\frac{1}{d(s,\bfC)} d(s,\bfC)=1,
$$
so the second coordinate of $\gpolar$ belongs to $\aleph_\bfC$. For
example, if $\bfS=\Rplus^2$ and we remove the cone consisting of
the axes through $0_2$, that is,
$\bfC=\{0\}\times [0,\infty) \cup [0,\infty)\times \{0\}$, then
$\aleph_\bfC=\{x \in \Rplus^2: x_1\wedge x_2=1\}$. The inverse map
$\gpolar{^{-1}} : (0,\infty)\times \aleph_\bfC \mapsto
\bfS\setminus \bfC$ is 
$$\gpolar{^{-1}} (r,a)=r a,\quad r\in (0,\infty),\,a\in\aleph_\bfC.$$

It is relatively easy to check that if ${A'} \subset (0,\infty)\times
\aleph_\bfC $ is bounded away from $\bfC'=\{0\}\times \aleph_\bfC$,
then
$\gpolar^{-1}({A'}) $ is bounded away from $\bfC$. On
$(0,\infty)\times \aleph_{{\bfC}}$ adopt the metric 
$$d'\bigl(
(r_1,a_1),(r_2,a_2)\bigr)=|r_1-r_2|\vee d_{\aleph_\bfC}
(a_1,a_2),$$
where
$d_{\aleph_\bfC}(a_1,a_2)$ is an appropriate metric on $\aleph_\bfC$. Suppose
$d'({A'}, \{0\}\times \aleph_{\bfC}) =\epsilon >0.$  This means 
\begin{align*}
\epsilon=&\inf_{
\substack{(r_1,a_1)\in {A'}\\ a_2 \in \aleph_\bfC}} 
d'\bigl((r_1,a_1), (0,a_2)\bigr)
\end{align*}
and setting $a_2=a_1$ this is $\inf_{(r_1,a_1) \in \bfD'} r_1$. 
We conclude that $(r,a)\in {A'}$ implies $r\geq
\epsilon$. Since $\gpolar^{-1}({A'})=\{r a : (r,a) \in {A'}\}, $
we have in $\bfS\setminus \bfC$, remembering that $\bfC$ is assumed to
be a cone,
\begin{align*}
d(\{r a:(r,a)\in {A'}\}, \bfC)
&=\inf_{(r,a) \in {A'}} d(r a,\bfC)=\inf_{(r,a) \in {A'}}
d(r a,r \bfC)\\
&=\inf_{(r,a) \in {A'}} r d(a,\bfC) \geq \epsilon \cdot 1.
\end{align*}
The last line uses \eqref{eq:metricScales} and the definition of
$\aleph_{\bfC}$.

The hypotheses of Theorem \ref{mapthm} are verified so we get the
following conclusion about $\gpolar$.
\begin{cor}\label{cor:gpolarCont}
Suppose $\bfS $ is a complete, separable metric space such that
\eqref{eq:metricScales} holds, scalar multiplication is defined
 and supposed $\bfC $ is a
closed cone. Then $\mu_n \to \mu_0$ in $\bfM(\bfS\setminus \bfC)$
implies
$$\mu_n\circ \gpolar^{-1} \to \mu_0\circ \gpolar^{-1} $$
in 
$\bfM(  (0,\infty)\times \aleph_\bfC)$. The converse holds as well.
\end{cor}

The converse is proven in a similar way.

{\bf Remark} on \eqref{eq:metricScales}: As mentioned,  \eqref{eq:metricScales}
 holds if the metric is defined by a norm  on the space $\bfS$.
Thus, \eqref{eq:metricScales} holds for a Banach space and on
$\Rplus^p$, $\mathbb{C}[0,1]$ with sup-norm and $\bfD([0,1], \R)$ with Skorohod
metric. It fails in $\Rplus^\infty$.

\subsection{Reducing $\Rplus^\infty$ convergence to finite dimensional
  convergence.}\label{subsec:reduce}
Corollary \ref{cor:projConv} shows when convergence in 
$\bfM(\Rplus^\infty \setminus \bfC)$ implies convergence in 
$\bfM(\Rplus^p \setminus \projp(\bfC))$. Here is a circumstance where
the converse is true and reduces the problem of convergence in
infinite dimensional space to finite dimensions.

\begin{thm}\label{th:fidiConvEnough}
Supppse for every $p\geq 1$, that the closed set $\bfC 
\subset \Rplus^\infty $
 satisfies $\projp(\bfC)$ is closed
in $\Rplus^p$  and 
\begin{equation}\label{eq:weirdCondit}
(z_1,\dots,z_p)\in \projp(\bfC) \quad \text{implies} \quad
(z_1,\dots,z_p,0_\infty) \in \bfC.
\end{equation}
Then
$\mu_n \to \mu_0$ in $\bfM(\Rplus^\infty \setminus \bfC)$ if and only if for all
$p\geq 1$ such that $\Rplus^p\setminus \projp(\bfC) \neq \emptyset$
we have
\begin{equation}\label{eq:fidiConv}
\mu_n \circ \projp^{-1} \to \mu_0\circ \projp^{-1} \end{equation}
 in 
$\bfM(\Rplus^p \setminus \projp(\bfC))$.
\end{thm}

{\bf Remark} on condition \eqref{eq:weirdCondit}: The condition says
take an infinite sequence $z $ in $\bfC$, truncate it to $z_{|p}\in \Rplus^p$,
and then make it infinite again by filling in zeros for all the
components beyond the $p$th. The result must still be in
$\bfC$. Examples:
\begin{enumerate}
\item $\bfC=\{0_\infty\}$.
\item Pick an integer $j \geq 1$ and define
\begin{equation}\label{eq:Fj}
\bfC_{\leq j}=\{x \in \Rplus^\infty: \sum_{i=1}^\infty \epsilon_{x_i}
(0,\infty) \leq j\}, \end{equation}
where recall
$\epsilon_x (A)=1$,   if $x \in A,$ and $0,$  if $x \in A^c$.
So $\bfC_{\leq j}$ consists of sequences with at most $j$
positive components. Truncation and then insertion of zeros does not
increase the number of positive components so $\bfC_{\leq j}$ is invariant
under the operation implied by \eqref{eq:weirdCondit}.
\end{enumerate}

\begin{proof}
Suppose $\bfC$ satisfies \eqref{eq:weirdCondit} and 
\eqref{eq:fidiConv} holds. Suppose $f \in \C(\Rplus^\infty \setminus
\bfC)$ and without loss of generality suppose $f$ is uniformly
continuous with modulus of continuity
$$\omega_f(\eta)=\sup_{\substack{(x,y)\in \Rplus^\infty \setminus
    \bfC\\d_\infty(x,y)<\eta}}  |f(x)-f(y)|.
$$
There exists $1>\delta>0$ such that $d_\infty(x,\bfC)<\delta$ implies
$f(x)=0$. Observe,
\begin{equation}\label{eq:missing}
d_\infty \bigl((x_{|p},0_\infty), x\bigr) \leq
\sum_{j=p+1}^\infty 2^{-j} =2^{-p}.\end{equation}
Pick any $p$ so large that $2^{-p}<\delta/2 $ and define
$$g(x_1,\dots,x_p)=f({x_1,\dots,x_p},0_\infty).$$
Then we have 
\begin{enumerate}
\item[(a)] From \eqref{eq:missing},
  $$|f(x)-g(x_{|p})|=|f(x)-f(x_{|p},0_\infty)|\leq
  \omega_f(2^{-p}).$$
\item[(b)] $g \in \C(\Rplus^p \setminus \projp (\bfC) )$ and $g$ is
  uniformly continuous.
\end{enumerate}
To verify that the support of $g$ is positive distance  away from
$\projp(\bfC)$, suppose $d_p$ is the $L_1$ metric on $\Rplus^p$  and
$d_p\bigl((x_1,\dots,x_p),\projp(\bfC)\bigr) <\delta/2.$  Then there
is $(z_1,\dots,z_p) \in \projp(\bfC)$ such that 
$d_p\bigl(((x_1,\dots,x_p),(z_1,\dots,z_p) \bigr)<\delta$.  But then
if $z \in \bfC$ with $z_{|p}=(z_1,\dots,z_p) $, we have, since
$(z_1,\dots,z_p,0_\infty) \in \bfC$ by \eqref{eq:weirdCondit},
\begin{align*}
d_\infty  &\bigl( (x_1,\dots,x_p,0_\infty),
(z_1,\dots,z_p,0_\infty)\bigr) 
=\sum_{i=1}^p \frac{|x_i-z_i|\wedge 1}{2^i}\\& \leq \sum_{i=1}^p
|x_i-z_i|\wedge 1 \leq \sum_{i=1}^p
|x_i-z_i| \\&=d_p\bigl( (x_1,\dots,x_p), (z_1,\dots,z_p)\bigr) <\delta,
\end{align*}
and therefore $d_p\bigl((x_1,\dots,x_p),\projp(\bfC)\bigr) <\delta/2$ implies
$$g(x_1,\dots,x_p)=f(x_1,\dots,x_p,0_\infty) =0.$$
So the support of $g$ is bounded away from $\projp(\bfC)$ as claimed.

Now write
\begin{align}
\mu_n(f)-\mu_0(f)&=[\mu_n(f)-\mu_n(g\circ \projp )] 
+[\mu_n(g\circ \projp ) -\mu_0(g\circ \projp )]\nonumber \\
&+[\mu_0(g\circ \projp )- \mu_0(f)] =A+B+C. \label{eq:decomp}
\end{align} 
From \eqref{eq:fidiConv}, since $g\circ \projp \in
\C(\Rplus^p\setminus \projp(\bfC))$, we have
$$B=\mu_n(g\circ \projp ) -\mu_0(g\circ \projp ) \to 0$$ as $n\to\infty.$

How to control A? For $x \in \Rplus^\infty \setminus \bfC$, if 
$d_\infty \bigl( (x_1,\dots,x_p,0_\infty),\bfC\bigr)< \delta/2$,
then $f((x_1,\dots,x_p,0_\infty) =0$ and also
$d_\infty \bigl(x,\bfC) \leq d_\infty (x, (x_{|p},0_\infty)
\bigr) + d_\infty \bigl((x_{|p},0_\infty),\bfC\bigr)
<2^{-p}+\delta/2<\delta $ so $f(x)=0.$
Therefore, on 
$$\Lambda =\{x \in \Rplus^\infty \setminus \bfC: d_\infty \bigl( (x_{|p},0),\bfC\bigr)<\delta/2\}
$$
both $f$ and $g\circ \projp $ are zero. Set 
$$\Lambda^c =(\Rplus^\infty
\setminus \bfC) \setminus \Lambda =\{x \in \Rplus^\infty \setminus
\bfC: d_\infty \bigl( (x_{|p},0),\bfC\bigr)\geq \delta/2\}.
$$ Then we have
\begin{align*}
|\mu_n(f)-\mu_n(g\circ \projp ) |\leq & \int |f-g\circ \projp | d\mu_n\\
= & \int_{\Lambda^c} |f-g\circ \projp | d\mu_n\\
\leq & \mu_n (\Lambda^c) \omega_f(2^{-p}),
\end{align*} and similarly for dealing with term C, we would have $
|\mu_0(f)-\mu_0(g\circ \projp ) | \leq \mu_0 (\Lambda^c) \omega_f(2^{-p})$.

Owing to finite dimensional convergence \eqref{eq:fidiConv} and
\eqref{eq:decomp}, we have
$$\limsup_{n\to\infty} |\mu_n(f) -\mu_0(f)| \leq 2
\omega_f(2^{-p})\mu_0(\Lambda^c) +0.$$ Since $\Lambda^c$ is bounded
away from $\bfC$, $\mu_0(\Lambda^c)<\infty$ and since the inequality
holds for any $p$ sufficiently large such that $2^{-p}<\delta, $ we
may let $p\to\infty$ to get $\mu_n(f) \to \mu_0(f).$
\end{proof}

{\bf Remark:} The proof shows \eqref{eq:fidiConv} only needs to hold
for all $p\geq p_0$. For example, if 
$$
\bfC=\bfC_j = \{x \in \Rplus^\infty: \sum_{i=1}^\infty
\epsilon_{x_i}(0,\infty) \leq j\},
$$ then
$$\projp(\bfC_j) = \{(x_1,\dots,x_p) \in \Rplus^p: \sum_{i=1}^p
\epsilon_{x_i}(0,\infty) \leq j\},$$
and for $p<j,$ $\projp (\bfC_j) =\Rplus^p$ and 
$\Rplus^p \setminus \projp (\bfC_j) =\emptyset. $  However, it
suffices for the result to hold for all $p\geq j$.

\subsection{Comparing $\bfM$-convergence on $\bfS\setminus \bfC$ with
  vague convergence when $\bfS$ is compactified.}\label{subsub:compareCompact}

This continues the discussion of Section \ref{subsec:MvsVague}.  Conventionally
\citep{resnickbook:2007} regular variation on $[0,\infty)^p$ has been defined on
the punctured compactified space $[0,\infty]^p\setminus
\{0_p\}$. 
This solves the problem of how to make tail regions relatively
compact. However, as discussed in 
\cite{das:mitra:resnick:2013},
 when deleting more than $\{0_p\}$, this 
approach causes problems with the convergence to types lemma and also
because certain natural regions are no longer relatively compact. 
The issue arises when
there is mass on the lines through $\infty_p$, something that is
impossible for regular variation on $[0,\infty]^p \setminus
\{0_p\}$. The following discussion amplifies what is in 
\cite{das:mitra:resnick:2013}.

Suppose $\bfC $ is closed in ${[0,\infty]^p}$ and set 
$$\bfC_0=\bfC\cap [0,\infty)^p,\quad \Omega=[0,\infty]^p \setminus \bfC,
\quad \Omega_0=[0,\infty)^p \setminus \bfC_0.$$
Examining the definitions we see that,
\begin{itemize}
\item $\Omega_0\subset \Omega.$
\item $\Omega\setminus \Omega_0={\bfC}^c \cap \bigl( [0,\infty]^p\setminus
[0,\infty)^p\bigr)$
\end{itemize}

\begin{prop}\label{prop:vagueAndM}
Suppose for every $n\geq 0$ that $\mu_n \in \bfM_+(\Omega)$ and $\mu_n
$ places no mass on the lines through $\infty_p$:
\begin{align}
&\mu_n(\bigl([0,\infty]^p \setminus [0,\infty)^p\bigr)\cap \bfC^c)=0. \label{eq:noMassCond}\\
 \intertext{Then}
&\mu_n \stackrel{v}{\to} \mu_0 \quad \text{ in }
\bfM_+(\Omega),\label{eq:vagueNomass}\\
\intertext{if and only if the restrictions to the space without the lines through
  $\infty_p$ converge:}
&\mu_{n0}:=\mu_n(\cdot \cap \Omega_0) \to  \mu_0(\cdot \cap \Omega_0)=:\mu_{00} \quad \text{
  in }
\bfM(\Omega_0). \label{eq:MconvNomass}
\end{align}
\end{prop}

\begin{proof}
Given \eqref{eq:MconvNomass}, let $f\in \C_K^+(\Omega)$. Then the
restriction to $\Omega_0$ satisfies
$f|_{\Omega_0} \in \C(\Omega_0)$ so $$\mu_n(f)=\mu_{n0}(f|_{\Omega_0} )
\to \mu_{00}(f|_{\Omega_0} )=\mu_0(f),$$
so $\mu_n\stackrel{v}{\to} \mu_0$
 in $M_+(\Omega)$.

Conversely, assume \eqref{eq:vagueNomass}. Suppose $B \in
\scrS(\Omega_0)$ and $\mu_{00}(\partial_{\Omega_0} B)=0$, where
$\partial_{\Omega_0} B$ is the set of boundary points of $B$ in
$\Omega_0$. This implies $\mu_0(\partial_\Omega B)=0$ since
$$\partial_\Omega (B) \subset \partial_{\Omega_0} B \cup \Bigl(\bigl(
\Omega\setminus \Omega_0\bigr) \cap \bfC\Bigr).$$
Therefore $\mu_n(B) \to \mu_0(B)$ and because of
\eqref{eq:noMassCond}, 
$\mu_{n0}(B)\to\mu_{00}(B)$ which proves \eqref{eq:MconvNomass}.
\end{proof}

\subsection{Regular variation on $\Rplus^p$ and $\Rplus^\infty$}\label{subsec:regVar} 
For this section, either $\bfS$ is  $\Rplus^p$ or $\Rplus^\infty$ and 
$\bfC$ is a closed cone; then  $\bfS \setminus \bfC$ is still a
cone. Applying Definition \ref{rvdefmeas},
a random element $X$ of $\bfS \setminus \bfC$ has a regularly
varying distribution if for some regularly varying function $b(t)
\to\infty$, as $t\to\infty$,
$$tP[X /b(t) \in \cdot\,] \to \nu(\cdot) \quad \text{ in } \bfM(\bfS
\setminus \bfC),$$
for some limit measure $\nu \in \bfM(\bfS\setminus \bfC)$. In $\Rplus^p$, if $\bfC=\{0_p\}$ or if
\eqref{eq:noMassCond} holds, this definition is the same as the one
using vague convergence on the compactified space.

\subsubsection{The iid case: remove $\{0_\infty\}$.}\label{eg:iid} 
Suppose $X=(X_1,X_2,\dots )$ is iid with non-negative components,
each of which has a regularly varying distribution on $(0,\infty)$
satisfying $$P[X_1>tx]/P[X_1>t]\to x^{-\alpha},\quad \text{ as
}t\to\infty,\,x>0,\,\alpha >0.$$
 Equivalently, as $t\to \infty$,
\begin{equation}\label{eq:1dimRV}
tF(b(t)\cdot): =tP[X_1/b(t) \in \cdot \,]\to \nu_\alpha (\cdot) \quad
\text{ in }\bfM((0,\infty)),
\end{equation}
where $\nu_\alpha(x,\infty)=x^{-\alpha},\,  \alpha>0.$
Then in $\bfM(\Rplus^\infty \setminus \{0_\infty\})$, we have
\begin{align}\label{eq:rvRinfty}
\mu_t((dx_1,dx_2,\dots )&:=tP[X/b(t) \in  (dx_1,dx_2,\dots )\,] \nonumber \\
&\to  \sum_{l=1}^\infty \prod_{i\neq l}
\epsilon_0 (dx_i) \nu_\alpha (dx_l) 
=:\mu^{(0)}(dx_1,dx_2,\dots),
\end{align}
and the limit measure concentrates on
$$\bfC_{=1}=\{x\in\Rplus^\infty: \sum_{i=1}^\infty \epsilon_{x_i}
    \bigl((0,\infty)\bigr) =1\},$$
the sequences with exactly one component positive. Note
$\{0_\infty\} \cup \bfC_{=1} =:\bfC_{\leq 1}$, the sequences
with at most one component positive, is closed.

To verify \eqref{eq:rvRinfty}, note from Theorem \ref{th:fidiConvEnough}, it suffices to verify finite dimensional
convergence since $\{0_\infty\}$ satisfies \eqref{eq:weirdCondit},
so it suffices to prove as $t\to\infty$, for $p\geq 1$,
\begin{align}
\mu_t\circ \projp^{-1} &((dx_1,\dots,dx_p)):=tP[(X_1,\dots,X_p)/b(t)
\in (dx_1,\dots,dx_p)\,]\nonumber \\
&\to \mu^{(0)} \circ 
\projp^{-1} ((dx_1,\dots,dx_p))= \sum_{l=1}^p \prod_{i\neq l}
\epsilon_0 (dx_i) \nu_\alpha (dx_l) , \label{eq:damnFiDi's}\end{align}
in $\bfM(\Rplus^p \setminus \{0_p\}$. Since neither $\mu_t\circ
\projp^{-1}$ nor $\mu^{(0)}\circ \projp^{-1}$ place mass on the lines
through $\infty_p$, $\bfM$-convergence and vague convergence are the
same and then \eqref{eq:damnFiDi's} follows from the binding lemma in
\cite[p. 228, 210]{resnickbook:2007}.

Applying the operator $\cumsum$ and  Corollary \ref{cor:cumsumRV} to \eqref{eq:rvRinfty}, gives,
\begin{align}
tP[\cumsum(X)/b(t) &\in (dx_1,dx_2,\dots )]
\to \mu^{(0)} \circ \cumsum^{-1} ((dx_1,dx_2,\dots ))  \nonumber\\
&= \sum_{l=1}^\infty \prod_{i=1}^{l-1}
\epsilon_0 (dx_i) \nu_\alpha (dx_l) \prod_{i=l+1}^\infty
\epsilon_{x_l}(dx_i) \quad \text{ in }\bfM(\Rplus^\infty \setminus \{0_\infty\}))
\label{eq:genBrei}
\end{align}
where the limit concentrates on non-decreasing sequences with one jump
and the size of the jump is governed by $\nu_\alpha$. Then applying
the operator $\projp$ we get by Corollary \ref{cor:projConv},
\begin{align}
tP[(X_1,&X_1+X_2,\dots ,\sum_{i=1}^pX_k)/b(t) \in
(dx_1,dx_2,\dots,dx_p )]\nonumber \\
&\to \mu^{(0)} \circ \cumsum^{-1} \circ \projp^{-1}((dx_1,dx_2,\dots,dx_p ))  \nonumber\\
&= \sum_{l=1}^p \prod_{i=1}^{l-1}
\epsilon_0 (dx_i) \nu_\alpha (dx_l) \prod_{i=l+1}^p
\epsilon_{x_l}(dx_i) \quad \text{ in }\bfM(\Rplus^p \setminus \{0_p\})),
\label{eq:genBreifidi}
\end{align}
giving an elaboration of the one big jump heuristic saying that summing independent risks which
have the same heavy tail results {in}
a tail risk which is the number of summands times the individual tail
risk; for example, see \cite[p. 230]{resnickbook:2007}.
In 
particular, applying the projection from $\Rplus^p \setminus\{0_p\} \mapsto (0,\infty)$ defined by $T:(x_1,\dots,x_p)\mapsto x_p$ gives by Corollary \ref{cor:variant1} that,
\begin{align}
tP[\sum_{i=1}^p X_i>b(t)x]\to & \mu^{(0)}\circ \cumsum^{-1}\circ
\projp^{-1}\circ T^{-1} (x,\infty)\label{eq:bigjumpsum}\\
=&p\nu_\alpha
(x,\infty)=px^{-\alpha}. \nonumber
\end{align}
The projection $T$ is uniformly continuous but also Theorem
\ref{mapthm} applies to $T$ since for $y>0$, 
$T^{-1}(y,\infty)=\{(x_1,\dots,x_p):x_p>y\}$ is at positive distance
from $\{0_p\}$.

The above discussion could have been carried out with minor
modifications without the iid
assumption by  assuming \eqref{eq:1dimRV} and 
$$ P[X_j>x]/P[X_1>x] \to c_j>0,\quad j\geq 2.$$

\subsubsection{The iid case; remove more; hidden regular variation.}\label{subsubsec:HRV}
We now investigate how to get past the one big jump heuristic by using
hidden regular variation.
For $j\geq 1$, set 
\begin{align}\bfC_{=j}=&\{x\in\Rplus^\infty: \sum_{i=1}^\infty \epsilon_{x_i}
    \bigl((0,\infty)\bigr) =j\},\nonumber \\
\bfC_{\leq j}=&\{x\in\Rplus^\infty: \sum_{i=1}^\infty \epsilon_{x_i}
    \bigl((0,\infty)\bigr) \leq j\} = \bfC_{\leq (j-1)} \cup \bfC_{=j} ,\label{eq:Fleq}
\end{align}
so that
$\bfC_{\leq j}$ is closed. We imagine an infinite sequence of
reductions of the state space with scaling adjusted at each step. This
is suggested by the previous discussion. On $\bfM(\Rplus^\infty 
\setminus \{0_\infty\})$, the limit measure $\mu^{(0)}$
concentrated on $\bfC_{=1}$, a small part of the potential state space. Remove $
 \{0_\infty\} \cup \bfC_{=1} =\bfC_{\leq 1}$ and on
 $\bfM(\Rplus^\infty\setminus \bfC_{\leq 1})$ seek a new convergence
 using adjusted scaling $b(\sqrt t)$. We get in
 $\bfM(\Rplus^\infty\setminus \bfC_{\leq 1})$ as $t\to\infty$,  
\begin{align}
\mu_t^{(1)}(dx_1,dx_2,\dots )&=
tP[X/b(\sqrt t) \in (dx_1,dx_2,\dots )]
\to \mu^{(1)}\bigl( (dx_1,dx_2,\dots)\bigr)\nonumber \\&
:=\sum_{l<k} \Bigl(\prod_{j \notin \{l,k\} }\epsilon_0 (dx_j)\Bigr)\nu_\alpha
  (dx_l)\nu_\alpha (dx_k)
\end{align}
which concentrates on $\bfC_{=2}$.
In general, we find that in  $\bfM(\Rplus^\infty\setminus \bfC_{\leq j})$ as $t\to\infty$,  
\begin{align}
\mu_t^{(j)}(dx_1,dx_2,&\dots )=
tP[X/b(t^{1/(j+1}) \in (dx_1,dx_2,\dots )]
\to \mu^{(j)}\bigl( (dx_1,dx_2,\dots)\bigr)\nonumber \\
&
:=\sum_{i_1<i_2<\dots <i_{j+1}} \Bigl( \prod_{j \notin \{i_1,\dots,i_{j+1}\}
}\epsilon_0 (dx_j) \Bigr)
\nu_\alpha
  (dx_{i_1}) \nu_\alpha (dx_{i_2})\dots \nu_\alpha
  (dx_{i_{j+1}})\label{eq:mutTomu}
\end{align}
which concentrates on $\bfC_{=(j+1)}$. This is an elaboration of
results in \cite{maulik:resnick:2005,
  mitra:resnick:2010a,mitra:resnick:2011hrv}. 
The result in
$\Rplus^\infty$ can be proven by reducing to $\Rplus^p$ by means of
Theorem \ref{th:fidiConvEnough} noting that $\bfC_{\leq j}$ satisfies
\eqref{eq:weirdCondit} and then observing that neither $\mu_t^{(j)}$ nor
$\mu^{(j)}$ puts mass on lines through $\infty_p$. It is enough to
show convergences of the following form: Assume $p\geq j$ and
$i_1<i_2< \dots < i_{j+1}$ and $y_l>0, \,l=1,\dots,j+1$ and
\begin{align*}
tP[X_{i_l}>b(t^{1/(j+1)}) y_l, \, l=1,\dots,j+1] &=\prod_{l=1}^{j+1}
t^{1/(j+1)}P[X_{i_l}>b(t^{1/(j+1)}) y_l]\\
& \to 
\prod_{l=1}^{j+1} \nu_\alpha (y_i,\infty)=\prod_{l=1}^{j+1}
y_l^{-\alpha}.
\end{align*}
A formal statement of the result and a proof relying on a convergence
determining class is given in the next Section
\ref{subsubsec:formal}. {Table \ref{table1} gives} a summary of the results in tabular
form.

\begin{table}
\centering
\begin{tabular}{|c|c|c|c|c|}
\hline
$j$&remove &scaling& $\mu^{(j)}$&support\\[2mm]
\hline \hline 
1& $\{0\}$&$b(t)$&$\sum_{l=1}^\infty \nu_\alpha (dx_l) \bigl[\prod_{i\neq l}
\epsilon_0(dx_i) \bigr]
 $  &axes\\[2mm]
\hline \\
2& axes & $b(\sqrt t)$ & 
$\displaystyle{\sum_{l,m} \nu_\alpha (dx_l) \nu_\alpha (dx_m)\Bigl[\prod_{i\notin \{l,m\} } }\epsilon_0(dx_i) \Bigr]
$ & 2-dim faces\\[2mm]
\hline \\
$\vdots $ &$\vdots $ &$\vdots $ &$\vdots $ &$\vdots $ \\[2mm] \hline
$m$ & $\bfC_{\leq(m-1)} $& $b(t^{\frac 1m})$ &
$\displaystyle{\sum_{(l_1,\dots,l_m)} } 
\displaystyle{\prod_{p=1}^m }\nu_\alpha (dx_{l_p}) 
\Bigl[
\displaystyle{\prod_{i\notin \{l_1,\dots,l_m\}  }}\epsilon_0(dx_i)\Bigr]
$ & $\bfC_{=m}  $\\
\hline
\end{tabular} \caption{An infinite number of coexisting regular
  variation properties.}
\label{table1}
\end{table}

Proposition
\ref{prop:cumsumCont} and Corollary \ref{cor:variant1}
allow application of $\cumsum$   to get 
\begin{align}
\mu_t^{(j)} \circ \cumsum^{-1}(dx_1,dx_2,\dots )=&
tP[\cumsum(X)/b(t^{1/(j+1}) \in (dx_1,dx_2,\dots )]\nonumber \\
\to & \mu^{(j)} \circ \cumsum^{-1}\bigl( (dx_1,dx_2,\dots)\bigr) \label{eq:oneMore}
\end{align}
in $\bfM(\cumsum(\Rplus^\infty)\setminus \cumsum(\bfC_{\leq j}))$.
Note $\cumsum(\Rplus^\infty)=:\mathbb{R}_+^{\infty\,\uparrow}$ is the
set of non-decreasing sequences and $\cumsum(\bfC_{\leq
  j})=:\bfS_{\leq j}$ is the set of 
non-decreasing sequences with at most $j$ positive jumps.
{Now apply the map $\projp$ to \eqref{eq:oneMore} 
to get a $p$-dimensional result for 
$(X_1,X_1+X_2,\dots,X_1+\dots +X_p)$ and the analogue of
 \eqref{eq:genBreifidi} is
\begin{align}
tP[(X_1,&X_1+X_2,\dots ,\sum_{i=1}^pX_k)/b(t^{1/(j+1)}) \in
(dx_1,dx_2,\dots,dx_p )]\nonumber \\
&\to \mu^{(j)} \circ \cumsum^{-1} \circ \projp^{-1}((dx_1,dx_2,\dots,dx_p ))  \label{eq:pdim}
\end{align}
in {$\bfM(\cumsum(\mathbb{R}_+^p) \setminus \projp (\bfS_{\leq j})
=\bfM(\mathbb{R}_+^{p\,\uparrow}) \setminus \projp (\bfS_{\leq j})$.}

When $j>1$, unlike the step leading to \eqref{eq:bigjumpsum}, we 
cannot apply the map $T:(x_1,\dots,x_p)\mapsto x_p$ to \eqref{eq:pdim} 
to get a marginal
result for $X_1+\dots + X_p$. Although $T$ is uniformly continuous,
Corollary \ref{cor:variant1} is not applicable since
$$T(\mathbb{R}_+^{p\,\uparrow}) \setminus T(\projp (\bfS_{\leq
  j}))=[0,\infty)\setminus [0,\infty)=\emptyset.$$



\subsubsection{The iid case; HRV; formal statement and proof.}\label{subsubsec:formal}
Recall $X=(X_l, l\geq 1)$ has iid components each of which has a
distribution with a regularly varying tail of index $\alpha>0$. Define
$\bfC_{\leq j}$ as in \eqref{eq:Fleq} and set $\bfO_j=\R_+^{\infty}\setminus\bfC_{\leq j}$.
The definition of $\mu_t$ and $\mu^{(j)}$ are given in \eqref{eq:mutTomu}.

\begin{thm}\label{thm:itehidrv}
For every $j\geq 1$ there is a nonzero measure $\mu^{(j)} \in\M_{\bfO_j}$ with support in $\bfC_{=(j+1)}$ such that $t\Prob[X/b(t^{1/{j+1}}) \in \cdot\,] \to \mu^{(j)}(\cdot)$ in $\M_{\bfO_j}$ as
$t\to\infty$. The measure $\mu^{(j)}$ is given in \eqref{eq:mutTomu}, or more formally,
\begin{align*}
\mu^{(j)} (A)=\sum_{(i_1,\dots,i_{j+1})}\int
I\Big\{\sum_{k=1}^{j+1}z_{k} e_{i_k}\in A\Big\}
\nu_\alpha (dz_1)\dots \nu_\alpha (dz_{j+1}),
\end{align*} 
where the components of $e_{i_k}$ are all zero except component $i_k$ whose value is $1$ and the indices $(i_1,\dots,i_{j+1})$ run through the ordered subsets of size $j+1$ of $\{1,2,\dots\}$.
\end{thm}

The proof of Theorem \ref{thm:itehidrv} uses a particular convergence determining class $\bfA_{\geq j}$ of subsets of $\bfO_j$. Let $\bfA_{\geq j}$ denote the set of sets $A_{m,i,a}$ for $m\geq j$, where
\begin{align*}
A_{m,i,a}=\{x\in\R_+^{\infty}:x_{i_k}>a_k \text{ for } k=1,\dots,m\},
\quad i_1<\dots<i_m, a_1,\dots,a_m>0.
\end{align*}

\begin{lem}\label{lem:convdetclass}
If $\mu_t,\mu\in\M_{\bfO_j}$ and $\lim_{t\to\infty}\mu_t(A)=\mu(A)$ for all $A\in\bfA_{\geq j}$ bounded away from $\bfC_j$ with $\mu(\partial A)=0$, then $\mu_t\to\mu$ in $\M_{\bfO_j}$ as $t\to\infty$.
\end{lem}

\begin{proof}
Consider the set of finite differences of sets in $\bfA_{\geq j}$ and note that this set is a $\pi$-system. 
Take $x\in\bfO_j$ and $\vep>0$. 
Since $x\in\bfO_j$ there are $i_1<\dots<i_j$ such that $x_{i_k}>0$ for each $k$.  
If $2^{-i_{j}}<\vep/2$ choose $m=i_j$. Otherwise, choose $m>i_j$ such that $2^{-m}<\vep/2$.
Take $\delta<\min\{\vep/2,\min\{x_k:x_k>0 \text{ and } k\leq m\}\}$ and set
\begin{align*}
B&=\{y\in\R_+^{\infty}:y_k\geq 0 \text{ if } x_k=0 \text{ and } y_k>x_k-\delta \text{ otherwise for }k\leq m\}\\
B'&=\{y\in\R_+^{\infty}:y_k>\delta \text{ if } x_k=0 \text{ and } y_k>x_k+\delta \text{ otherwise for }k\leq m\}.
\end{align*}
Then $B,B'\in \bfA_{\geq j}$, $B'$ is a proper subset of $B$, and $z\in B\setminus B'$ implies that $d(z,x)<\delta\sum_{k=1}^m 2^{-k}+\vep/2<\vep$, i.e.~that $z\in B_{x,\vep}$. Moreover,
\begin{align*}
(B\setminus B')^{\circ}=\{y\in\R_+^{\infty}: y_k\in J(x_k) \text{ for }k\leq m\},
\end{align*}
where $J(x_k)=[0,\delta)$ if $x_k=0$ and $J(x_k)=(x_k-\delta,x_k+\delta)$ if $x_k\neq 0$. 
Finally, 
$\partial (B\setminus B')$ is the set of $y\in\R_+^{\infty}$ such that $y_k\in [\max\{0,x_k-\delta\},x_k+\delta]$ for all $k\leq m$ and $y_k=\delta$ or $y_k=x_k\pm \delta$ for some $k\leq m$. In particular, there is an uncountable set of $\delta$-values, for which the boundaries $\partial (B\setminus B')$ are disjoint, satisfying the requirements. Therefore $\delta$ can without loss of generality be chosen so that $\mu(\partial (B\setminus B'))=0$. 
The separability of $\R_+^{\infty}$ implies (cf. the proof of Theorem 2.3 in \cite{billingsley:1999}) that each open set is a countable union of $\mu$-continuity sets of the form $(B\setminus B')^{\circ}$.  
The same argument as in the proof of Theorem 2.2 in \cite{billingsley:1999} therefore shows that $\liminf_{t\to\infty}\mu_t(G)\geq\mu(G)$ for all open $G\subset\bfO_j$ bounded away from $\bfC_j$.
Any closed set $F\subset\bfO_j$ bounded away from $\bfC_j$ is a subset of some $A\in\bfA_{\geq j}$. By the same argument as above, we may without loss of generality take $A$ such that $\mu(\partial A)=0$. The set $A\setminus F$ is open and therefore
\begin{align*}
\mu(A)-\limsup_{t\to\infty}\mu_t(F)=\liminf_{t\to\infty}\mu_t(A\setminus F)\geq \mu(A\setminus F)=\mu(A)-\mu(F),
\end{align*} 
i.e.~$\limsup_{t\to\infty}\mu_t(F)\leq\mu(F)$.
The conclusion follows from Theorem \ref{portthm}(iii).
\end{proof}

\begin{proof}[Proof of Theorem \ref{thm:itehidrv}]
For any $m>j$ and $a_1,\dots,a_m>0$,
\begin{align*}
\lim_{t\to\infty}c(t)^j\Prob(X\in tA_{j,i,a})=\prod_{k=1}^j a_k^{-\alpha}=\mu_j(A_{j,i,a})
\quad\text{and}\quad
\lim_{t\to\infty}c(t)^j\Prob(tA_{j+1,i,a})=0.
\end{align*} 
Therefore, the support of $\mu$ is a subset of $\bfC_{j+1}\setminus \bfC_{j}$.
Notice that for $j\geq 1$
\begin{align*}
(\bfC_{j+1}\setminus\bfC_{j})\cap\R_+^p=\cup_{i_1<\dots<i_{j}} \{(\lambda_1e_{i_1},\dots,\lambda_{j}e_{i_{j}});\lambda_1,\dots,\lambda_{j}>0\},
\end{align*}
where the indices $i_1,\dots,i_{j}$ run through the ordered subsets of size $j$ of $\{1,2,\dots\}$.
\end{proof}

\subsubsection{Poisson points as random elements of
  $\mathbb{R}_+^\infty$.}\label{subsubsec:PoissonPts}
Considering Poisson points provides a variant to the iid case and
leads naturally to considering regular variation of the distribution
of a L\'evy process with regularly varying measure.

Suppose $\nu \in \mathbb{M}(0,\infty)$ and 
$x\mapsto\nu(x,\infty)$ is regularly varying at infinity with index $-\alpha<0$. Let $Q(x)=\nu([x,\infty))$ and define
$Q^{\leftarrow}(y)=\inf\{t>0:\nu([t,\infty))<y\}$. Then the function
$b$ given by $b(t)=Q^{\leftarrow}(1/t)$ satisfies
$\lim_{t\to\infty}t\nu(b(t)x,\infty)=x^{-\alpha}$. It follows that $b$
is regularly varying at infinity with index $1/\alpha$. 

Let $\{E_n, n\geq 1\}$ be iid standard exponentially distributed random
variables so that if $\{\Gamma_n, n \geq 1\}:=\cumsum\{E_n, n\geq 1\}$,
we get points of a homogeneous Poisson process of rate
1. Transforming \cite[p. 121]{resnickbook:2007}, we find
$\{Q^\leftarrow (\Gamma_n), n\geq 1\}$ are points of a Poisson process
with mean measure $\nu$, written in decreasing order.

Define the following subspaces of $\R_+^{\infty}$: 
\begin{align}
\R_+^{\infty\,\downarrow}&=\{x\in\R_+^{\infty}:x_1\geq x_2\geq \dots
\},\nonumber \\
\bfH_{=j}&=\{x\in \R_+^{\infty\,\downarrow}:x_j>0,x_{j+1}=0\},\nonumber \\
\bfH_{\leq j}&=\{x\in \R_+^{\infty\,\downarrow}:x_{j+1}=0\}, \quad
\bfO_j=\R_+^{\infty\,\downarrow} \setminus \bfH_{\leq j}, \label{eq:notationPoisson}
\end{align}
with the usual meaning of multiplication by a scalar. 
So $\bfH_{\leq 0}=\{0_\infty\}$ and $\R_+^{\infty\,\downarrow}$ are sequences with decreasing,
non-negative components and $\bfH_{\leq j}$ are decreasing sequences
such that components are 0 from the $(j+1)st$ component
onwards. Furthermore, for each $j\geq 1$, $\bfH_{\leq j}$ is
closed. To verify the closed property, suppose $\{x(n), n\geq 1\}$
is a sequence in $\bfH_{\leq j}$ and $x(n) \to x(\infty)$ in the
$\R_+^\infty$ metric. This means componentwise convergence so for the
$m$th component convergence, where $m>j$, as $n\to \infty$,
$0=x_m(n) \to x_m(\infty)$ and $x(\infty)$ is 0 beyond the $j$th
component. The monotonicity of the components for each $x(n) $ is
preserved by taking limits. Hence $\bfH_{\leq j}$ is closed. 

Analogous to \eqref{eq:rvRinfty}, we claim
\begin{align*}
t\Prob[\bigl(Q^\leftarrow (\Gamma_l)/b(t), l\geq 1\bigr) \in \cdot \,]
\to \mu^{(1)}(\cdot), 
\end{align*}
in $\bfM(\bfO_0)$ as $t\to\infty$, where 
\begin{align*}
\mu^{(1)}(dx_1\times dx_2 \times \dots)=\nu_\alpha(dx_1)1_{[x_1>0]} \prod_{l=2}^\infty \epsilon_0(dx_l).
\end{align*}
To verify this, it suffices to prove finite
dimensional convergence and for the biggest component and $x>0$,
\begin{align*}
tP[Q^\leftarrow (\Gamma_1)/b(t)>x]&=tP[\Gamma_1\leq Q(b(t)x)]=t(1-e^{-Q(b(t)x)})\\
&\sim tQ(b(t)x) \to x^{-\alpha} =\nu_\alpha (x,\infty).
\end{align*}
For the first two components, let PRM($\nu$) be a Poisson counting
function with mean measure $\nu$ and for $x>0,\,y>0$,
\begin{align*}
tP[Q^\leftarrow (\Gamma_1)/b(t)>x,\,Q^\leftarrow (\Gamma_2)/b(t)>y]
\leq tP[\text{PRM}(\nu) (b(t)(x\wedge y, \infty) \geq 2]
\end{align*}
and writing $p(t)=\nu (b(t)(x\wedge y, \infty))$, we have
\begin{align*}
tP[\text{PRM}(\nu) (b(t)(x\wedge y, \infty) \geq 2]
&=t(1-e^{-p(t)} -p(t)e^{-p(t)})\\
&\leq t(p(t)-p(t)e^{-p(t)}) \leq  tp^2(t) \to 0.
\end{align*}
The conclusion now follows from Lemma \ref{lem:convdetclass} by observing that we have shown convergence for the sets in a convergence determining class.

Similarly, we claim 
\begin{align*}
t\Prob[\bigl(Q^\leftarrow (\Gamma_l)/b(t^{1/2}), l\geq 1\bigr) \in \cdot \,]
\to \mu^{(2)}(\cdot)
\end{align*}
in $\bfM(\bfO_1)$ as $t\to\infty$, where 
\begin{align*}
\mu^{(2)}(dx_1\times dx_2 \times \dots)=\nu_\alpha(dx_1)\nu_\alpha(dx_2)1_{[x_1\geq x_2>0]} \prod_{l=3}^\infty \epsilon_0(dx_l).
\end{align*}
Straightforward computations show that the distribution of $(\Gamma_1,\Gamma_2)=(E_1,E_1+E_2)$ satisfies
\begin{align*}
\Prob(\Gamma_1\leq z,\Gamma_2\leq w)=
\left\{\begin{array}{ll}
1-e^{-z}-ze^{-w}, & z<w,\\
1-e^{-w}-we^{-w}, & z\geq w.\\
\end{array}\right.
\end{align*}
Notice that, for $x>y>0$,
\begin{align*}
&\Prob[Q^\leftarrow (\Gamma_1)/b(t^{1/2})>x,\,Q^\leftarrow (\Gamma_2)/b(t^{1/2})>y]\\
&\quad=\Prob[\Gamma_1\leq Q(b(t^{1/2})x),\Gamma_2\leq Q(b(t^{1/2})y)]\\
&\quad=1-e^{-Q(b(t^{1/2})x)}-Q(b(t^{1/2})x)e^{-Q(b(t^{1/2})y)}\\
&\quad\sim Q(b(t^{1/2})x)-Q(b(t^{1/2})x)^2/2+O(Q(b(t^{1/2})x)^3)\\
&\quad\quad - Q(b(t^{1/2})x)\Big(1-Q(b(t^{1/2})y)+O(Q(b(t^{1/2})y)^2)\Big)
\end{align*}
In particular, it is a straightforward exercise in calculus to verify that for $x>y>0$
\begin{align*}
&\lim_{t\to\infty}t\Prob[Q^\leftarrow (\Gamma_1)/b(t^{1/2})>x,\,Q^\leftarrow (\Gamma_2)/b(t^{1/2})>y]\\
&\quad=x^{-\alpha}y^{-\alpha}-x^{-2\alpha}/2\\
&\quad=\mu^{(2)}(z\in \R^{\infty\,\downarrow}:z_1>x,z_2>y).
\end{align*}
Similar computations show that, for $y>x>0$,
\begin{align*}
&\lim_{t\to\infty}t\Prob[Q^\leftarrow (\Gamma_1)/b(t^{1/2})>x,\,Q^\leftarrow (\Gamma_2)/b(t^{1/2})>y]\\
&\quad=y^{-2\alpha}/2\\
&\quad=\mu^{(2)}(z\in \R^{\infty\,\downarrow}:z_1>x,z_2>y).
\end{align*}
Moreover, for $x>0,\,y>0,\,z>0$,
\begin{align*}
&tP[Q^\leftarrow (\Gamma_1)/b(t^{1/2})>x,\,Q^\leftarrow (\Gamma_2)/b(t^{1/2})>y,\,Q^\leftarrow (\Gamma_3)/b(t^{1/2})>z]\\
&\quad\leq tP[\text{PRM}(\nu) (b(t^{1/2})(x\wedge y\wedge z, \infty) \geq 3]
\end{align*}
and writing $p(t)=\nu (b(t^{1/2})(x\wedge y\wedge z, \infty))$, we have
\begin{align*}
tP[\text{PRM}(\nu) (b(t^{1/2})(x\wedge y\wedge z, \infty) \geq 3]
&=t(1-e^{-p(t)} -p(t)e^{-p(t)}-p(t)^2e^{-p(t)}/2)\\
&\sim t(p(t)^3/3!+o(p(t)^3))
\end{align*}
as $t\to\infty$. Hence, $\lim_{t\to\infty}tP[\text{PRM}(\nu) (b(t^{1/2})(x\wedge y\wedge z, \infty) \geq 3]=0$.

As in the iid case described by Theorem \ref{thm:itehidrv} and
\eqref{eq:mutTomu}, we have an infinite number of regular variation
properties co-existing.

\begin{thm}\label{thm:PoissonPtsConv} 
For the Poisson points $\{Q^\leftarrow (\Gamma_l), l\geq 1\bigr \}$,
for every $j \geq 1$, we have  
\begin{align}\label{eq:hidePoisson}
tP\bigl[\bigl({Q^\leftarrow (\Gamma_l)} /{b(t^{1/j})}, l\geq 1\bigr) \in
\cdot \,\bigr] \to \mu^{(j)}(\cdot),
\end{align}
in $\bfM_{\bfO_{j-1}}$ as $t\to\infty$, where $ \mu^{(j)}$ is a measure concentrating on $\bfH_{=j}$ given by
\begin{align}\label{eq:nuj}
\mu^{(j)}(dx_1,dx_2,\dots)=\prod_{i=1}^j \nu_\alpha (dx_i)1_{[x_1\geq
  x_2\geq \dots \geq x_j>0]} \prod_{i=j+1}^\infty \epsilon_0(dx_i).
\end{align}
\end{thm}

\begin{proof}
The explicit computations above, and similarly for $j\geq 3$, together with an application of Lemma \ref{lem:convdetclass} yields the conclusion.
\end{proof}

\section{Finding the hidden jumps of a L\'evy process}\label{sec:hrvLevy}
In this section we consider a real valued L\'evy process $X=\{X_t, t\geq 0\}$ as a random element of $\bfD:=\bfD([0,1],\R)$, the space of real valued c\`adl\`ag
functions on $[0,1]$. We metrize $\bfD$ with the usual Skorohod
metric 
$$d_{\text{sk}}(x,y)=\inf_{\lambda \in \Lambda}\|\lambda -e\|\vee \|x\circ \lambda -y\|,$$
where $x,y \in \bfD,$ $\lambda$ is a non-decreasing
homeomorphism of $[0,1]$ onto itself, $\Lambda$ is the set of all such
homeomorphisms, $e(t)=t$ is the identity, and
$\|x\|=\sup_{t\in [0,1]} |x(t)|$ is the sup-norm. 
The space $\bfD$ is not complete under the metric $d_{\text{sk}}$, but
there is an equivalent metric under which $\bfD$ is complete \cite[page 125]{billingsley:1999}. Therefore, the space $\bfD$ fits into the framework presented in Section \ref{secconv} and we may use the Skorohod metric to check continuity of mappings.

For simplicity we suppose $X$ has only positive
jumps and its L\'evy measure $\nu$ concentrates on $(0,\infty)$. Suppose
$x\mapsto\nu(x,\infty)$ is regularly varying at infinity with index
$-\alpha<0$. Let $Q(x)=\nu([x,\infty))$ and define
$Q^{\leftarrow}(y)=\inf\{t>0:\nu([t,\infty))<y\}$. Then the function
$b$ given by $b(t)=Q^{\leftarrow}(1/t)$ satisfies
$\lim_{t\to\infty}t\nu(b(t)x,\infty)=x^{-\alpha}$ {and} $b$
is regularly varying at infinity with index $1/\alpha$. 
{It is shown in \cite{hult:lindskog:2005SPA,hult:lindskog:2007} that with scaling
function $b(t)$, the distribution of $X$ is regularly
varying on $\bfD\setminus \{0\}$ with a limit measure
concentrating on functions which are constant except for one jump.} 
Where did the other L\'evy process jumps go? Using weaker scaling  and
 biting more out of $\bfD$ than just the zero-function $0$,
allows recovery of  the other jumps.

The standard Ito representation \cite{
bertoin:1996,
kyprianou:2006,
applebaum:2004}
 of $X$ is 
\begin{align*}
X_t=ta+B_t+\int_{|x|\leq 1}x[N([0,t]\times dx)-t\nu(dx)]+\int_{|x|>1}xN([0,t]\times dx),
\end{align*}
where $B$ is standard Brownian motion independent of the Poisson
random measure $N$ on $[0,1]\times (0,\infty)$ with mean measure
$\text{Leb} \times \nu$. 
Referring to the discussion preceding \eqref{eq:notationPoisson},
$\{Q^\leftarrow (\Gamma_n), n\geq 1\}$ are points written in
decreasing order of a Poisson random measure on $(0,\infty)$ with mean
measure $\nu$ and by augmentation
\citep[p. 122]{resnickbook:2007}, we can represent 
$$N=\sum_{l=1}^{\infty}\epsilon_{(U_l,Q^{\leftarrow}(\Gamma_l))},$$
where $(U_l,l\geq 1)$ are iid standard uniform random variables
independent of $\{\Gamma_n\}$.

The L\'evy-Ito decomposition allows $X$ to be decomposed into the sum
of two independent L\'evy processes, 
\begin{equation}\label{eq:Levy-Ito}
X=\widetilde{X}+J,
\end{equation}
 where $J$ is
a compound Poisson process of large jumps bounded
from below by $1$, and $\widetilde{X}=X-J$ is a L\'evy process of
small jumps that are bounded from above by $1$. The compound Poisson process can
be represented as the random sum
$J=\sum_{l=1}^{N_1}Q^{\leftarrow}(\Gamma_l)1_{[U_l,1]}$, where
$N_1=N([0,1]\times [1,\infty))$. 

Recall the notation in \eqref{eq:notationPoisson} for 
$\R_+^{\infty\,\downarrow}$, 
$\bfH_{=j}$ and $\bfH_{\leq j}$ and the result in Theorem
\ref{thm:PoissonPtsConv}. 
We seek to convert a statement like \eqref{eq:hidePoisson} into a statement about $X$. 
The first step is to augment \eqref{eq:hidePoisson} with a sequence of
iid standard uniform random varables. The uniform random variables
will eventually serve as jump times {for the L\'evy process.}
The following result is an immediate consequence of Theorem \ref{thm:PoissonPtsConv}.

\begin{prop}\label{thm:ordseqconv}
Under the given assumptions on $\nu$ and $Q$,  for $j\geq 1$, 
\begin{align}
t\Prob\Big[
\bigl((Q^{\leftarrow}(\Gamma_l)/b(t^{1/j}),l\geq 1),(U_l,l\geq 1)\bigr)\in
\,\cdot\,\Big] \to (\mu^{(j)}\times L)(\cdot) \label{eq:augConv}
\end{align}
in $\M((\R_+^{\infty\,\downarrow}\setminus \bfH_{\leq j-1})\times [0,1]^{\infty})$ as $t\to\infty$,
where $L$ is Lebesgue measure on $[0,1]^{\infty}$ and $\mu^{(j)}$ concentrates on $\bfH_{=j}$ and is given by \eqref{eq:nuj}.
\end{prop}

Think of \eqref{eq:augConv} as regular variation on 
 the product space $\R_+^{\infty\,\downarrow}\times [0,1]^{\infty}$
when  multiplication by a scalar is defined as $(\lambda,(x,y))\mapsto (\lambda x,y)$.

Recall $\nu_\alpha $ is the Pareto measure on $(0,\infty)$
satisfying $\nu_\alpha (x,\infty)=x^{-\alpha}$, for $x>0$, and we
denote by $\nu_\alpha^j$ product measure generated by $\nu_\alpha$
with $j$ factors.  
For $m\geq 0$, let $\bfD_{\leq m}$ be the subspace of the Skorohod
space $\bfD$ consisting of nondecreasing step functions with at most
$m$ jumps {and define} {$A_m$ as
\begin{align}\label{eq:Am}
A_m 
&= \{ (x,u)  \in \R_+^{\infty\,\downarrow}\times [0,1]^{\infty} \\
&\qquad : u_i \in (0,1) \mbox{ for } 
1 \leq i \leq m; 
u_i \neq u_j \mbox{ for } i \neq j, 1 \leq i,j \leq m \}. \nonumber
\end{align}
 Let} $T_m$ be the map 
\begin{equation}\label{eq:Tm}
T_m:A_m \mapsto \bfD \text{ defined by }
T_m (x,u)= \sum_{i=1}^m x_i1_{[u_i,1]},
\end{equation}
and we think of $T_m$ as mapping a jump size sequence and {a 
sequence of distinct jump times} into a step function in $\bfD_{\leq m}\subset \bfD$.
Our approach 
applies $T_m$ to the convergence in
\eqref{eq:augConv} to get a sequence of regular
variation properties  of the distribution of $X$.
Whereas in Section \ref{subsubsec:HRV}, we could rely on uniform continuity of
CUMSUM, $T_m$ is not uniformly continuous and hence the  mapping
Theorem \ref{mapthm} must be used and its hypotheses verified. 
We will prove the following.

\begin{thm}\label{thm:Xconv}
Under the regular variation assumptions on $\nu$ and $ Q$, for $j\geq 1$, 
\begin{align}\label{eq:hrvX}
t\Prob \bigl[{X}/{b(t^{1/j})} \in \,\cdot\, \bigr]\to &
(\mu^{(j)}\times L)\circ T_j^{-1}(\cdot)\\= &
\E\Big[\nu_{\alpha}^j\Big\{y\in (0,\infty)^j:\sum_{i=1}^j
y_i1_{[U_i,1]}\in\cdot\Big\}\Big] \nonumber  
\end{align}
in $\M(\bfD\setminus \bfD_{\leq j-1})$ as $t\to\infty$.
\end{thm}
{The first expression after taking the limit in \eqref{eq:hrvX}
 follows from the mapping Theorem \ref{mapthm} and the second
from applying $T_j$  to \eqref{eq:augConv}  and then using Fubini to hold the
integration with respect to Lebesgue measure $L$ outside as an expectation. }
\vspace{-.2in}
\begin{proof}
Here is the outline; more detail is given in the next section. We
prove convergence using  Theorem \ref{portthm} (iii).
Take $F$ and $G$ closed and open sets respectively in $\bfD$ that are
bounded away from $\bfD_{\leq j-1}$. 
Take $\delta>0$ small enough so that also $F_\delta =\{x\in\bfD: \dsk (x,F)\leq \delta\}$ is bounded away from $\bfD_{\leq j-1}$. Then 
\begin{align}
tP[ X/b(t^{1/j}) \in F]
&= t\Prob\Big[X \in b(t^{1/j})F,\sup_{s\in [0,1]}|\widetilde{X}_s|\leq b(t^{1/j})\delta\Big]\nonumber\\
&\quad + t\Prob\Big[X \in b(t^{1/j})F,\sup_{s\in [0,1]}|\widetilde{X}_s|> b(t^{1/j})\delta\Big]\nonumber \\
&\leq t\Prob [J \in b(t^{1/j})F_\delta ]+tP[\sup_{s\in
  [0,1]}|\widetilde{X}_s|> b(t^{1/j})\delta]. \label{eq:smallAway}
\end{align}
The L\'evy process $\widetilde{X} $ has all moments finite and does
not contribute asymptotically. Application of  Lemmas \ref{lem:Jconv}
and \ref{lem:smalljumpsdontmatter}, and letting $\delta\downarrow 0$
gives 
\begin{align*}
\limsup_{t\to\infty}tP[X /b(t^{1/j}) \in F]
\leq (\mu^{(j)}\times L)\circ T_j^{-1}(F).
\end{align*}

To deal with the lower bound using open $G$, 
take $\delta>0$ small enough so that 
$${G^{-\delta}:=\bigl((G^c)_\delta\bigr)^c=\{x\in G: d_{\text{sk}}
(x,y)<\delta \text{ implies }y \in G\}}
$$
 is nonempty and
bounded away from $\bfD_{\leq j-1}$. Then 
\begin{align*}
tP[X /b(t^{1/j})  \in G]
&\geq t\Prob\Big[J \in b(t^{1/j})G^{-\delta},\sup_{s\in [0,1]}|\widetilde{X}_s|\leq b(t^{1/j})\delta\Big]\\
&= t\Prob\Big[J \in b(t^{1/j})G^{-\delta}\Big]\Prob\Big[\sup_{s\in
  [0,1]}|\widetilde{X}_s|\leq b(t^{1/j})\delta\Big]. 
\end{align*}
Applying Lemmas \ref{lem:Jconv}
and \ref{lem:smalljumpsdontmatter}
and letting $\delta\downarrow 0$ gives
\begin{align*}
\liminf_{t\to\infty}tP[X/b(t^{1/j}) \in G] \geq (\mu^{(j)}\times L)\circ T_j^{-1}(G).
\end{align*}
\end{proof}

\subsection{Details}
We now provide more detail for the proof of Theorem
\ref{thm:Xconv}. 


In the decomposition \eqref{eq:Levy-Ito}, the process
$\widetilde{X}$ represents small jumps that should not affect
asymptotics. We make
this precise with the next Lemma.

\begin{lem}\label{lem:smalljumpsdontmatter}
For $j\geq 1$, and any $\delta >0$,
$$\limsup_{t\to\infty}t\Prob\Big[\sup_{s\in [0,1]}|\widetilde{X}_s|>b(t^{1/j})\vep\Big]=0.$$ 
\end{lem}
\begin{proof}
We rely on Skorohod's inequality for L\'evy processes
\cite{breiman:1992}, \cite[Section 7.3]{resnick:1999book}.
For $a>0$,
$$\Prob\Big[\sup_{s\in [0,1]}|\widetilde{X}_s|
>2a\Big]\leq (1-c)^{-1}\Prob[|\widetilde{X}_1|>a],$$
where $c=\sup_{s\in [0,1]}\Prob[|\widetilde{X}_s|>a]$. Thus, since
$\widetilde{X}_1$ has all moments finite, for any $m>1$, 
\begin{align*}
t\Prob\Big[\sup_{s\in [0,1]}|\widetilde{X}_s|>b(t^{1/j}) \delta\Big]
&\leq {t}(1-c(t))^{-1}P[|\widetilde{X}_1|>b(t^{1/j}) \delta/2]\\
&\leq {t}(1-c(t))^{-1} \frac{\E |\widetilde{X}_1|^m}{b^m(t^{1/j} )(\delta/2)^m}.
\end{align*} 
For large enough $m$, $t/b^m(t^{1/j}) \to 0$ as $t\to \infty$ and 
\begin{align*}
c(t):=&\sup_{s\in [0,1]} P[|\widetilde{X}_s>b(t^{1/j}) \delta/2]
\leq \sup_{s\in [0,1]}\frac{\E |\widetilde{X}_s|^m}{b^m(t^{1/j})(\delta/2)^m}\\
=&\sup_{s\in [0,1]} \frac{s^m\E |\widetilde{X}_1|^m}{b^m(t^{1/j})(\delta/2)^m} 
\leq \frac{\E |\widetilde{X}_1|^m}{b^m(t^{1/j})(\delta/2)^m} \to 0,
\end{align*} 
as $t\to\infty$ since $b(t) \to \infty$.
\end{proof}

\begin{lem}\label{lem:Jconv}
For $j\geq 1$,
$t\Prob[J \in b(t^{1/j})\,\cdot\,]\to (\mu^{(j)}\times L)\circ T_j^{-1}(\cdot)$
in $\M(\bfD\setminus \bfD_{\leq j-1})$ as $t\to\infty$.
\end{lem}
\begin{proof}
We {apply}
 Theorem \ref{portthm} (iii).

{\sc Construction of the lower bound for open sets:}
Let $G\subset \bfD$ be open and bounded away from $\bfD_{\leq
  j-1}$. {This implies that functions in $G$ have no fewer than $j$ jumps.}
Recall that $\Gamma_l=E_1+\dots+E_l$, where the $E_k$s are iid standard exponentials.
Take $M\geq j$ and notice that
\begin{align*}
t\Prob\Big[&\sum_{l=1}^{N_1}Q^{\leftarrow}(\Gamma_l)1_{[U_l,1]}\in b(t^{1/j})G\Big]\\
&\geq t\Prob\Big[\sum_{l=1}^{N_1}Q^{\leftarrow}(\Gamma_l)1_{[U_l,1]}\in b(t^{1/j})G,N_1\leq M\Big]\\
&=t\Prob\Big[\sum_{l=1}^{N_1}Q^{\leftarrow}(\Gamma_l)1_{[U_l,1]}\in b(t^{1/j})G,j\leq N_1\leq M\Big]\\
& \geq t\Prob\Big[\sum_{l=1}^{j}Q^{\leftarrow}(\Gamma_l)1_{[U_l,1]}\in b(t^{1/j})G^{\delta},\sum_{l=j+1}^{M}Q^{\leftarrow}(\Gamma_l)\leq b(t^{1/j})\delta{, Q^{\leftarrow}(\Gamma_{M+1}) < 1} \Big]\\
& \geq
t\Prob\Big[\sum_{l=1}^{j}\frac{Q^{\leftarrow}(\Gamma_l)}{b(t^{1/j})}1_{[U_l,1]}\in
G^{\delta},
M\frac{Q^{\leftarrow}(E_{j+1})}{b(t^{1/j})}\leq \delta{, Q^{\leftarrow}(\Gamma_{M+1} - \Gamma_{j+1}) < 1}\Big]\\
& \geq t\Prob\Big[\sum_{l=1}^{j}
\frac{Q^{\leftarrow}(\Gamma_l)}{b(t^{1/j})}1_{[U_l,1]}\in
b(t^{1/j})G^{\delta}\Big]\Prob\Big[MQ^{\leftarrow}(E_{j+1})\leq
b(t^{1/j}\delta\Big]
\\ 
&\qquad \qquad 
{\times \Prob\Big[Q^{\leftarrow}(\Gamma_{M+1} -
  \Gamma_{j+1}) < 1 \Big]}\\
&
\geq t\Prob\Big[\Bigl((\frac{Q^{\leftarrow}(\Gamma_l) }{b(t^{1/j})}
,l\geq 1),(U_l,l\geq 1)\Bigr)\in  T_j^{-1}(G^{\delta})\Big]\Prob\Big[MQ^{\leftarrow}(E_{j+1})\leq b(t^{1/j})\delta\Big]\\
&\qquad \qquad {\times \Prob\Big[Q^{\leftarrow}(\Gamma_{M+1} - \Gamma_{j+1}) < 1 \Big]}
\end{align*}
Let $t\to\infty$ and apply Theorem \ref{portthm} (iii) to
\eqref{eq:augConv} so the $\liminf$ of the first factor above has  a
lower bound. As $t\to\infty$, the second factor approaches $1$. Let
$M\to \infty$ and the third factor also approaches $1$.
Let $\delta\downarrow 0$ and we obtain
\begin{align*}
\liminf_{t\to\infty}t\Prob\Big[\sum_{l=1}^{N_1}Q^{\leftarrow}(\Gamma_l)1_{[U_l,1]}\in
b(t^{1/j})G\Big]&\geq (\mu^{(j)}\times L)\circ T_j^{-1}(G).
\end{align*}

{\sc Construction of the upper bound for closed sets:}
Let $F\subset \bfD$ be closed and bounded away from $\bfD_{\leq j-1}$. Take $\beta\in (0,1)$ close to $1$ and let 
\begin{align*}
M_t=\sum_{l=1}^{N_1}1_{(b(t^{1/j})^{\beta},\infty)}(Q^{\leftarrow}(\Gamma_l)).
\end{align*}
Choose $\delta>0$ small enough so that $F_\delta   :=\{x\in\bfD:
d(x,F)\leq \delta\}$ 
 is bounded away
from $\bfD_{\leq j-1}$.
Then
\begin{align*}
t\Prob\Big[&\sum_{l=1}^{N_1}Q^{\leftarrow}(\Gamma_l)1_{[U_l,1]}\in b(t^{1/j})F\Big]\\
&= t\Prob\Big[\sum_{l=1}^{N_1}Q^{\leftarrow}(\Gamma_l)1_{[U_l,1]}\in b(t^{1/j})F,
\sum_{l=M_t+1}^{N_1}Q^{\leftarrow}(\Gamma_l)\leq b(t^{1/j})\delta\Big]\\
&\qquad
+t\Prob\Big[\sum_{l=1}^{N_1}Q^{\leftarrow}(\Gamma_l)1_{[U_l,1]}\in
b(t^{1/j})F,\sum_{l=M_t+1}^{N_1}Q^{\leftarrow}(\Gamma_l)>b(t^{1/j})\delta\Big].\\
&\leq t\Prob\Big[\sum_{l=1}^{M_t}Q^{\leftarrow}(\Gamma_l)1_{[U_l,1]}\in b(t^{1/j})F_\delta \Big]
+t\Prob\Big[\sum_{l=M_t+1}^{N_1}Q^{\leftarrow}(\Gamma_l)>b(t^{1/j})\delta\Big]\\
\intertext{Decompose the first summand according to whether $M_t\leq
  j$ or $M_t\geq j+1$. Notice $M_t<j$ is incompatible with 
$\sum_{l=1}^{M_t}Q^{\leftarrow}(\Gamma_l)1_{[U_l,1]}\in
b(t^{1/j})F_\delta$ since $F_\delta$ is bounded away from $\bfD_{\leq
  j-1}$. Thus we get the upper bound}
&\leq t\Prob\Big[\sum_{l=1}^{j}Q^{\leftarrow}(\Gamma_l)1_{[U_l,1]}\in b(t^{1/j})F_\delta \Big]
+t\Prob[M_t\geq j+1]\\
&\quad\quad+t\Prob\Big[\sum_{l=M_t+1}^{N_1}Q^{\leftarrow}(\Gamma_l)>b(t^{1/j})\delta\Big].
\end{align*}

We now show that the second and third of the three terms above vanish as $t\to\infty$. 
Firstly,
{the definition of $M_t$ implies that $Q^{\leftarrow}(\Gamma_l)
  \leq b(t^{1/j})^\beta$ for $M_t+1 \leq l \leq N_1$.} 
Thus,
\begin{align*}
t\Prob\Big[\sum_{l=M_t+1}^{N_1}Q^{\leftarrow}(\Gamma_l)>b(t^{1/j})\delta\Big] &{\leq tP[(N_1 - M_t) b(t^{1/j})^\beta >b(t^{1/j})\delta]}\\ 
&\leq tP[N_1>b(t^{1/j})^{1-\beta}\delta].
\end{align*}
The right-hand side converges to $0$ as $t\to\infty$ since the tail
probability has a Markov bound of $tE(N_1)^p/[b(t^{1/j})^{1-\beta}
\delta ]^p$ for any $p$.
Secondly,
\begin{align*}
\Prob[M_t\geq j+1]&\leq \Prob[Q^{\leftarrow}(\Gamma_{j+1})>b(t^{1/j})^{\beta}]\\
&\leq \Prob[\Gamma_{j+1}\leq \nu([b(t^{1/j})^{\beta},\infty))]\\
&\leq \Prob[\max(E_1,\dots,E_{j+1})\leq \nu([b(t^{1/j})^{\beta},\infty))]\\
&= \Prob[E_1\leq \nu([b(t^{1/j})^{\beta},\infty))]^{j+1}.
\end{align*}
Since $P[E_1\leq y]\sim y$ as $y\downarrow 0$, and since $\nu([x,\infty))$ is regularly varying at infinity with index $-\alpha$ and $b$ is regularly varying at infinity with index $1/\alpha$, we find that 
\begin{align*}
\limsup_{t\to\infty}t\nu([b(t^{1/j})^{\beta},\infty))^{j+1}
=\limsup_{t\to\infty}L(t)t^{1-\beta(j+1)/j}
\end{align*}
for some slowly varying function $L$. In particular, choosing 
$\beta\in {(\frac{j}{j+1},1)}$ ensures that $\lim_{t\to\infty}t\Prob[M_t\geq j+1]=0$.

We now deal with the remaining term.
Since 
\begin{align*}
&t\Prob\Big[\sum_{l=1}^{j}Q^{\leftarrow}(\Gamma_l)1_{[U_l,1]}\in b(t^{1/j})F_\delta \Big]\\
&\quad= t\Prob\Big[((Q^{\leftarrow}(\Gamma_l),l\geq 1),(U_l,l\geq 1))\in b(t^{1/j})\circ T_j^{-1}(F_\delta )\Big] \\
&{\qquad+ t\Prob\Big[((Q^{\leftarrow}(\Gamma_l),l\geq 1),(U_l,l\geq 1))\in A_m^c \Big]}
\end{align*}
and, by Lemmas \ref{lem:contTm} and \ref{lem:Tm}, $T_j^{-1}(F_\delta
)$ is, if nonempty, closed and bounded away from $\bfH_{\leq
  j-1}\times [0,1]^{\infty}$, {Proposition \ref{thm:ordseqconv}
  and the fact that $A_m^c$ is a $\Prob$-null set} yield that  
\begin{align*}
\limsup_{t\to\infty}t\Prob\Big[\sum_{l=1}^{N_1}Q^{\leftarrow}(\Gamma_l)1_{[U_l,1]}\in
b(t^{1/j})F\Big]\leq (\mu^{(j)}\times L)\circ T_j^{-1}(F_\delta).
\end{align*}
Letting $\delta\downarrow 0$ shows that 
\begin{align*}
\limsup_{t\to\infty}t\Prob\Big[\sum_{l=1}^{N_1}Q^{\leftarrow}(\Gamma_l)1_{[U_l,1]}\in b(t^{1/j})F\Big]\leq (\mu^{(j)}\times L)\circ T_j^{-1}(F).
\end{align*}
We have thus shown that 
$\liminf_{t\to\infty}t\Prob[J \in b(t^{1/j})G]\geq (\mu^{(j)}\times L_j)\circ T_j^{-1}(G)$ 
and
$\limsup_{t\to\infty}t\Prob[J \in b(t^{1/j})F]\leq (\mu^{(j)}\times L_j)\circ T_j^{-1}(F)$ 
for all open $G$ and closed $F$ bounded away from $\bfD_{\leq j-1}$.
The conclusion follows from Theorem \ref{portthm}.
\end{proof}

Recall the definitions of $A_m$ and $T_m$ in \eqref{eq:Am} and \eqref{eq:Tm}.
\begin{lem}\label{lem:contTm}
For $m\geq 1$, $T_m : A_m \mapsto \bfD$ is continuous. 
\end{lem}
\begin{proof}
The projection {
\begin{align*}
A_m \ni (x,u) \mapsto ((x_1,\dots,x_m),(u_1,\dots,u_m)) \in \R_+^{m\,\downarrow} \times (0,1)^{m,\neq}
\end{align*}
where $(0,1)^{m,\neq} = \{ (u_1,\ldots, u_m) \in (0,1)^m : u_i \neq u_j \mbox{ for } i \neq j \}$, 
is continuous.} Since compositions of continuous functions are continuous, it remains to check that
\begin{align*}
\R_+^{m\,\downarrow} \times (0,1)^{m,\neq} \ni ((x_1,\dots,x_m),(u_1,\dots,u_m)) \mapsto \sum_{i=1}^m x_i1_{[u_i,1]} \in \bfD
\end{align*}
is continuous.
Take $(x,u) \in \R_+^{m\,\downarrow} \times (0,1)^{m,\neq}$. Then
there exists some $\delta>0$ such that, for
$(\widetilde{x},\widetilde{u}) \in \R_+^{m\,\downarrow} \times
(0,1)^{m,\neq}$, $d_{2m}((x,u),
(\widetilde{x},\widetilde{u}))<\delta$, where $d_{2m}$ is the usual
metric in $\R^{2m}$, implies that the components of $\widetilde{u}$
appear in the same order as do the components of $u$. If
$0=u_{(0)}<u_{(1)}<\dots u_{(m)} <u_{(m+1)}=1,$ with corresponding
notation for the ordered $\tilde u$'s, make sure 
$3\cdot \delta < \vee_{i=1}^{m+1} |u_{(i)} -u_{(i-1)}|\vee |\tilde u_{(i)}
-\tilde u_{(i-1)}|.$
Consider the piece-wise linear function $\lambda_l$ for which $\lambda_l(0)=0$, $\lambda_l(1)=1$, and $\lambda_l(u_i)=\widetilde{u}_i$ for each $i$. Notice that $\lambda_l$ is strictly increasing and satisfies $\|\lambda_l-e\|<\delta$. Therefore,
\begin{align*}
\sup_{t\in [0,1]}\Big|\sum_{i=1}^m x_i1_{[\lambda_l(u_i),1]}(t)-\sum_{i=1}^m \widetilde{x}_i1_{[\widetilde{u}_i,1]}(t)\Big|<\sum_{i=1}^m|x_i-\widetilde{x}_i|<m\delta.
\end{align*}
In particular, 
\begin{align*}
d_{\text{sk}}\Big(\sum_{i=1}^m x_i1_{[u_i,1]},\sum_{i=1}^m\widetilde{x}_i1_{[\widetilde{u}_i,1]}\Big)<m\delta,
\end{align*}
which shows the continuity.
\end{proof}

\begin{lem}\label{lem:Tm}
{Suppose $A\subset \bfD$ is} bounded away from $\bfD_{\leq j-1}$. 
{For  $m\geq j$,
if} $T_m^{-1}(A)$ is nonempty, then it is bounded away from $\bfH_{\leq j-1}\times [0,1]^{\infty}$.
\end{lem}
\begin{proof}
If $A\cap \bfD_{\leq m}=\es$, then $T_m^{-1}(A)=\es$. Therefore,
without loss of generality we may  take $A\subset \bfD_{\leq m}$. 
{Assume  $d_{\text{sk}} (A,\bfD_{\leq j-1})>\delta>0 $ 
and n}otice that $x\in\bfD_{\leq m}$ if and only if
\begin{align*}
x=\sum_{i=1}^m y_i1_{[u_i,1]} \text{ for } y_1\geq \dots \geq y_m\geq 0, u_i\in [0,1].
\end{align*}
{If $x\in A$, $\sum_{i=j}^m y_i>\delta$ as a consequence of 
 $d_{\text{sk}}(A,\bfD_{\leq j-1})>\delta$ and because the $y$'s are
 non-increasing, $y_j>\delta/(m-j+1)$. Consequently,  
\begin{align*}
T_m^{-1}(A)&\subset \Big\{ \bigl(x_i,i\geq1\bigr) \in \R_+^{\infty\,\downarrow} : x_j > \delta/(m-j+1) \Big\} \times [0,1]^\infty,
\end{align*}
and the latter set is  bounded away from $\bfH_{\leq j-1}\times
[0,1]^{\infty}$.}
\end{proof}


\bibliography{bibfile}
\end{document}